\documentclass[11pt,reqno]{amsart}
\usepackage[T1]{fontenc}
\usepackage{lmodern}
\usepackage[english]{babel}
\usepackage{graphicx}
\usepackage{amsmath}
\usepackage{amssymb}
\usepackage{mathrsfs}
\usepackage{bookmark}
\usepackage[numbers,sort&compress]{natbib}
\numberwithin{equation}{section}
\usepackage{amsthm}
\usepackage[a4paper]{geometry}
\newcommand\Item[1][]{%
  \ifx\relax#1\relax  \item \else \item[#1] \fi
  \abovedisplayskip=0pt\abovedisplayshortskip=0pt~\vspace*{-\baselineskip}}

\theoremstyle{plain}
 \newtheorem{theorem}{Theorem}[section]

\newtheorem{lemma}[theorem]{Lemma}
 \newtheorem{corollary}[theorem]{Corollary}
\theoremstyle{definition}
 \newtheorem{example}{Example}[section]
 
 \newtheorem{remark}[theorem]{Remark}
 \numberwithin{equation}{section}

\usepackage{caption}
\usepackage{tcolorbox}

\DeclareMathOperator{\arsinh}{arcsinh}
\DeclareMathOperator{\arcosh}{arccosh}

\usepackage{mathtools}
\DeclarePairedDelimiter\ceil{\lceil}{\rceil}
\DeclarePairedDelimiter\floor{\lfloor}{\rfloor}

\usepackage{pinlabel}
\usepackage{xcolor}

\definecolor{mygreen}{rgb}{0.0, 0.5, 0.0}
\definecolor{deepmagenta}{rgb}{0.8, 0.0, 0.8}
\definecolor{deepcarrotorange}{rgb}{0.91, 0.41, 0.17}
\definecolor{heartgold}{rgb}{0.5, 0.5, 0.0}

\newcommand{\xuparrow}[1]{%
  {\left\uparrow\vbox to #1{}\right.\kern-\nulldelimiterspace}
}
\definecolor{kellygreen}{rgb}{0.3, 0.73, 0.09}

\usepackage{subfigure}

\usepackage{BOONDOX-cal}

\usepackage{enumitem}

\begin{document}

\title[Short closed geodesics on cusped hyperbolic surfaces]{Short closed geodesics on cusped hyperbolic surfaces}
\author[Thi Hanh Vo]{Thi Hanh Vo} 
\address{Department of Mathematics, University of Luxembourg, Luxembourg}
\curraddr{}
\email{thihanh.vo@uni.lu}

\subjclass[2010]{Primary: 30F10. Secondary: 53C22}

\keywords{closed geodesics, hyperbolic surfaces}
\date{\today}

\begin{abstract} 
This article deals with the set of closed geodesics on complete finite type hyperbolic surfaces. 
For any non-negative integer $k$, we consider the set of closed geodesics that self-intersect at least $k$ times, and investigate those of minimal length.
The main result is that, if the surface has at least one cusp, their self-intersection numbers are exactly $k$ for large enough $k$.
\end{abstract}

\maketitle


\section{Introduction} 
Let $X$ be a complete hyperbolic surface of finite type.
For a non-negative integer $k$, we focus on the set of closed geodesics on $X$ with at least $k$ self-intersection points. 
By discreteness of the length spectrum, there exists a $\gamma_k$ which is shortest among these geodesics.
Our goal is to study the actual value of the self-intersection number of $\gamma_k$.

For $k=0$, the set under consideration is the set of all closed geodesics on $X$, so this amounts to asking how many times does the shortest closed geodesic (the \textit{systole}) of $X$ self-intersect. 
Unless $X$ is a thrice punctured sphere, the answer is 0. 
For $k=1$, a result by Buser \cite{MR2742784} states that the shortest non-simple closed geodesic on $X$ has one self-intersection point (it is a so-called \textit{figure eight geodesic}).
For $k \ge 2$, there are no exact values known.

Let $I(k,X)$ denote the maximum number of self-intersections of shortest closed geodesics on $X$ with at least $k$ self-intersections. 
By definition, $I(k,X) \ge k$.
Basmajian \cite{MR1152271} provided the first upper bounds on these numbers. He showed that they are bounded from above by constants depending on $k$ and the genus of $X$. 
The dependence on the topology is used to bound the lengths of curves in a pants decomposition via a theorem by Bers \cite{MR780038}.

Erlandsson and Parlier \cite{2016arXiv160900217E} were able to get rid of the topological data. They proved that for all $k \ge 2$, $I(k,X)$ is bounded from above by a universal linear function in $k$. 
In fact, this result holds true for general surfaces (those with non-abelian fundamental group, including infinite area or infinite type surfaces).
In the presence of cusps, they provided a sharper bound on $I(k,X)$, which this time depends on the geometry of $X$. 
This bound implies that the numbers of self-intersections behave asymptotically like $k$ for growing $k$ when $X$ has at least one cusp.

To the best of the author's knowledge, there is not a single surface $X$ for which the equality $I(k,X)=k$ is known to hold for all $k$. Conversely, there are no surfaces for which the equality is known to fail for at least one $k \ge 1$. 
We show that if $X$ has a cusp, the equality holds for all large enough $k$.
 
\begin{theorem}\label{Theorem1}
Let $X$ be a complete finite type hyperbolic surface with at least one cusp. There exists a constant $K(X)$ depending on $X$ such that for all $k \ge K(X)$, if $\gamma_k$ is a shortest closed geodesic with self-intersection number at least $k$, then 
\begin{enumerate}[label=\upshape(\arabic*),ref=\thetheorem (\arabic*)]
  \item\label{Theorem1Part1} $\gamma_k$ self-intersects exactly $k$ times,
  \item\label{Theorem1Part2} $\gamma_k$ is freely homotopic to $\kappa * \tau^k$, where $\tau$ is freely homotopic to a cusp,
and $\kappa$ is a non-contractible simple loop that is not freely homotopic to $\tau$.
\end{enumerate}
\end{theorem} 

The constant $K(X)$, which can be found explicitly in Section \ref{SectionProof}, comes from the systole length and the shortest orthogonal distances from horocycles to themselves.
In our proof, we show that in terms of length, it is more efficient to wind around a cusp to generate self-intersection points than by doing anything else.

To simplify, we will use the term $ba_c^k$ \textit{geodesic} to refer to the geodesic mentioned in Theorem \ref{Theorem1Part2}.   
Let $\mathcal{M}_{g,n}^{\mathcal{s}}$ denote the set of all cusped hyperbolic surfaces of genus $g$ with $n$ punctures or boundary components, that are not the thrice punctured sphere, and the systole lengths of which are at least $\mathcal{s}$, where $\mathcal{s} > 0$. 
For $S \in \mathcal{M}_{g,n}^{\mathcal{s}}$, by Theorem \ref{Theorem1}, there is a constant $K(S)$ depending on the systole length and the shortest orthogonal distances from horocycles to themselves, such that for all $k \ge K(S)$, we have the conclusions $(1)$ and $(2)$.
The systole length is bounded from below by $\mathcal{s}$, whereas the shortest orthogonal distances can be evaluated explicitly in terms of topological data via Bers' constant \cite{MR780038}.
Therefore, $K(S)$ can be expressed as a constant depending only on $\mathcal{s}$ and the topological data.

\begin{corollary}\label{Theorem2}
For all $S \in \mathcal{M}_{g,n}^{\mathcal{s}}$,
there exists a constant $K = K(g,n,\mathcal{s})$ such that for all $k \ge K$, if $\gamma_k$ is a shortest closed geodesic on $S$ with at least $k$ self-intersection points, then 
\begin{enumerate}[label=\upshape(\arabic*),ref=\thetheorem (\arabic*)]
  \item\label{Theorem2Part1} $\gamma_k$ self-intersects exactly $k$ times,
  \item\label{Theorem2Part2} $\gamma_k$ is a $ba_c^k$ geodesic.
\end{enumerate}
\end{corollary}

The article is organized as follows. 
In Section \ref{IntersectionThickPart}, we generalize some computations on lengths and self-intersection numbers of geodesic segments. 
Section \ref{SectionProof} is dedicated to the proof of Theorem \ref{Theorem1}.
In Section \ref{SectionCorollary}, we prove Corollary \ref{Theorem2} and present some remarks. 

\section{Preliminary computations}\label{IntersectionThickPart}
Throughout this paper, $X$ will be a complete finite type hyperbolic surface with at least one cusp. 
By abuse of language, we use the term curve to mean both a parameterized curve and the set of points that make up the image of a parameterized curve. 
Recall that a geodesic realizes the minimal number of self-intersections among all curves in its free homotopy class.
All our closed geodesics are considered primitive.

Denote by $\mathcal{G}_{k}(X)$ the set of closed geodesics on $X$ that self-intersect exactly $k$ times. 
Basmajian \cite{MR3380365, MR3065183} investigated the following quantity 
$$s_k(X) := \inf \{ \ell(\gamma) : \gamma \in \mathcal{G}_{k}(X) \}.$$
He showed \cite[Proposition 4.2]{MR3065183} that for $k \ge 2$, there exists a constant 
$$C(k,X) := 2 \arsinh(k) + d(X) + 1,$$ 
such that 
\begin{equation}\label{EQBasmajian}
s_k(X) \le C(k,X),
\end{equation}
where $d(X)$ is the shortest orthogonal distance from the length one horocycle boundary of a cusp to the horocycle itself.
Note that $\arsinh(k)$ is comparable to $\log(k)$, and hence so is $C(k,X)$.

Likewise, denote by $\mathcal{G}_{\ge k}(X)$ the set of closed geodesics on $X$ that self-intersect at least $k$ times. Consider the following quantity
$$s_{\ge k}(X) := \inf \{ \ell(\gamma) : \gamma \in \mathcal{G}_{\ge k}(X) \}.$$
Since $s_{\ge k}(X) \le s_k(X)$, inequality (\ref{EQBasmajian}) also holds for $s_{\ge k}(X)$. Thus, if $\gamma_k$ is a shortest geodesic in $\mathcal{G}_{\ge k}(X)$, then
\begin{equation}\label{EQBasmajianExtended}
\ell(\gamma_k) \le C(k,X).
\end{equation}

Given $\varepsilon > 0$, we recall that the \textit{$\varepsilon$-thick part} $X_{\varepsilon T}$ of $X$ is the subset of $X$ consisting of points with injectivity radius at least $\varepsilon$, and the \textit{$\varepsilon$-thin part} $X_{\varepsilon t}$ of $X$ is the set of points with injectivity radius less than $\varepsilon$.
A curve $\gamma \subset X$ can be decomposed into $\gamma_{\varepsilon T} := X_{\varepsilon T} \cap \gamma$ and $\gamma_{\varepsilon t}  := X_{\varepsilon t} \cap \gamma$. 
Erlandsson and Parlier managed to control the relationship between length and self-intersection of $\gamma_{\varepsilon T}$.
They showed \cite[Theorem 1.3]{2016arXiv160900217E} that for $\varepsilon \le \frac{1}{2}$, the self-intersection of $\gamma_T$ satisfies
\begin{equation}\label{EQErlandssonParlier}
\ell(\gamma_{\varepsilon T}) > \frac{\varepsilon}{12} \sqrt{i(\gamma_{\varepsilon T}, \gamma_{\varepsilon T})}.
\end{equation}

Now, let $\varepsilon := \min \{ \frac{1}{4}, \frac{\mathcal{s}}{2}\}$, where $\mathcal{s} = \mathcal{s}(X)$ is the systole length of $X$. 
Choose $K(X) \ge 2$ such that 
$$C(k,X) < \frac{\varepsilon}{12} \sqrt{k},$$ 
for all $k\ge K(X).$
Let $k \ge K(X)$ and let $\gamma_k$ be a shortest geodesic in $\mathcal{G}_{\ge k}(X)$. By inequalities (\ref{EQBasmajianExtended}) and (\ref{EQErlandssonParlier}), $\gamma_k$ must enter $X_{\varepsilon t}$, the $\varepsilon$-thin part of $X$. Indeed, if $\gamma_k$ is entirely contained in $X_{\varepsilon T}$, then $\gamma_k=\gamma_{\varepsilon T}$.
We would have
\begin{equation*}\label{EQGammaEnterEpsilonThin}
 \frac{\varepsilon}{12} \sqrt{k} \le \frac{\varepsilon}{12} \sqrt{i(\gamma_k, \gamma_k)}  \le \ell(\gamma_k) \le C(k,X) < \frac{\varepsilon}{12} \sqrt{k},
\end{equation*}
which is a contradiction. Thus, $\gamma_k$ must enter $X_{\varepsilon t}$.
Due to our choice of $\varepsilon$, $\gamma_k$ must enter a cusp of $X$. Hence, understanding the behavior of geodesic segments lying in cusp neighborhoods will be crucial.
In the following, we present some facts about these geodesic segments.

For every cusp $\mathcal{C}$ of $X$, we have cusp neighborhoods filled by embedded horocycles winding around the cusp.
Let $\mathcal{h} = \mathcal{h}(\mathcal{C})$ be the supremum of the lengths of embedded horocycles in $X$ winding around $\mathcal{C}$.  
Given $0 < h < \mathcal{h}$, we denote by $\mathcal{B}_{h}$ the cusp neighborhood with boundary given by an embedded horocycle of length $h$. 
Let $\beta \subset \mathcal{B}_{h}$ be a geodesic segment that starts on $\partial \mathcal{B}_{h}$, winds around the cusp, and returns to $\partial \mathcal{B}_{h}$.
We refer to this type of geodesic segments as a \textit{strand} (see Figure \ref{FigureAStrand}).

\begin{figure}[ht]
    \centering
    \includegraphics[width=0.20\textwidth]{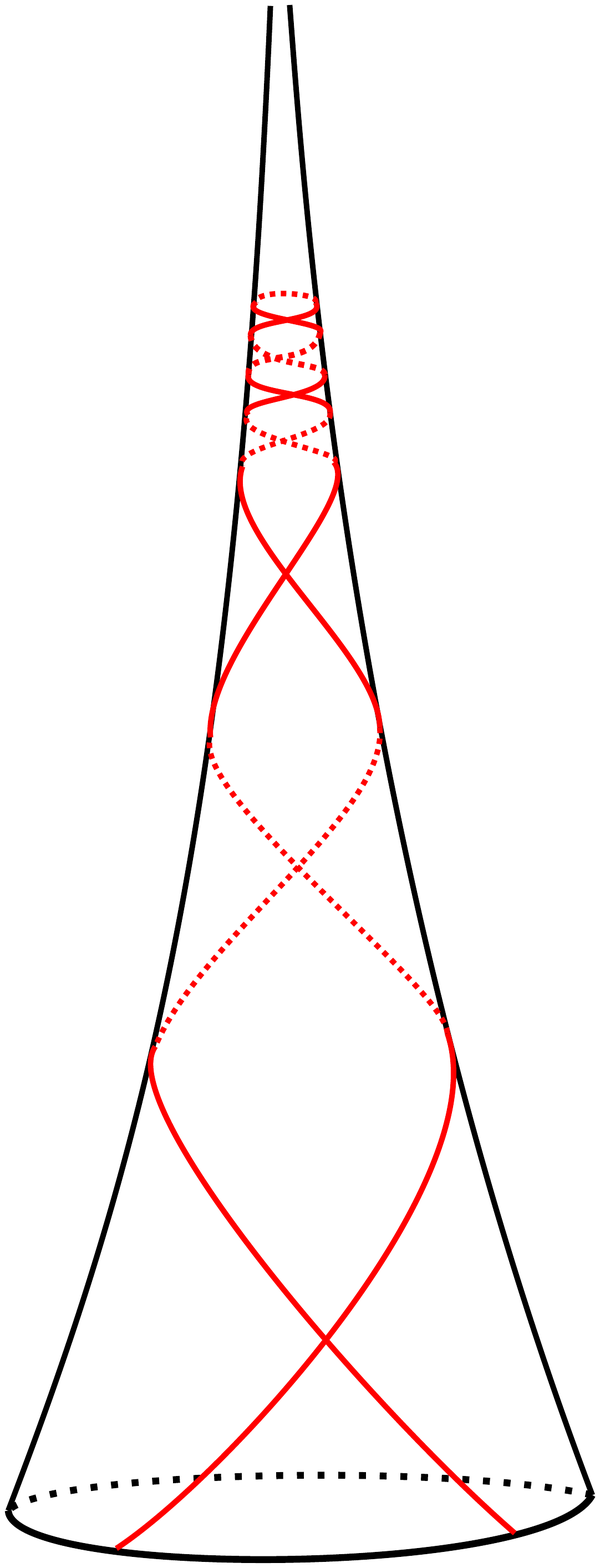}
    \caption{A strand.}
    \label{FigureAStrand}
\end{figure}

We define the \textit{winding number} $\omega(\beta)$ of the strand $\beta$ in $\mathcal{B}_h$ (with respect to $\mathcal{B}_h$) in the following way.
Every point of $\beta$ projects orthogonally to a well-defined point of the length $h$ horocycle.
The winding number of $\beta$ is given by the ceiling of the quotient of the length of the projection of $\beta$ (thought of as a parameterized segment) divided by $h$.

In the sequel, the symbol $\floor{\cdot}$ denotes the floor function. 
We begin with the following lemma which is a generalization of \cite[Lemma 4.1]{MR3065183}. 

\begin{lemma}\label{LemmaBasmajiangeneralization}
For $0 < h \le \mathcal{h}$, let $\beta \subset \mathcal{B}_{h}$ be a strand that starts on $\partial \mathcal{B}_{h}$, winds $\omega(\beta)$ times around the cusp, and returns to $\partial \mathcal{B}_{h}$. Then 
$$\omega(\beta) = i(\beta, \beta) + 1 = \floor*{ \frac{2}{h}\sinh \frac{\ell(\beta)}{2} }. $$
\end{lemma}

\begin{proof}
The procedure is essentially the same as the proof of \cite[Lemma 4.1]{MR3065183}. 
Upon lifting to the upper half-plane model for the hyperbolic plane, we can normalize so that $\mathcal{B}_{h}$ lifts to the region above the Euclidean line $y = \frac{1}{h}$, 
and the parabolic element identified with the element of the fundamental group that wraps once around $\mathcal{B}_{h}$ is $f: z \mapsto z+1$.
The lift of $\beta$ in the upper half-plane is a geodesic segment $\tilde{\beta}$ that is contained on an Euclidean semicircle orthogonal to the real axis. This semicircle can be further normalized so that its center is the origin (see Figure \ref{PictureLemmaBasmajianLength}).
Denote the endpoints of this semicircle by $-r$ and $r$. 
Write the endpoints of $\tilde{\beta}$ as $-r {\rm e}^{i \theta_0}$ and $r {\rm e}^{- i \theta_0}$, where $\theta_0$ is given by
$$\cos (\theta_0) = \frac{ \sqrt{ r^2 - \frac{1}{h^2} }}{r}.$$

\begin{figure}[ht]
\labellist
\small\hair 2pt
\pinlabel ${ \color{red} \tilde{\beta} }$ at 840 450
\pinlabel ${- r}$ at 0 20
\pinlabel ${r}$ at 1150 20
\pinlabel ${ 0 }$ at 580 0
\pinlabel {$ \frac{1}{h}$} at 540 200
\pinlabel {$r {\rm e}^{i \theta_0}$} at 1200 140
\pinlabel {$- r {\rm e}^{- i \theta_0}$} at -60 140
\pinlabel {$ ri $} at 580 610
\pinlabel ${\theta_0}$ at 850 58
\endlabellist
\centering
\includegraphics[width=0.4\textwidth]{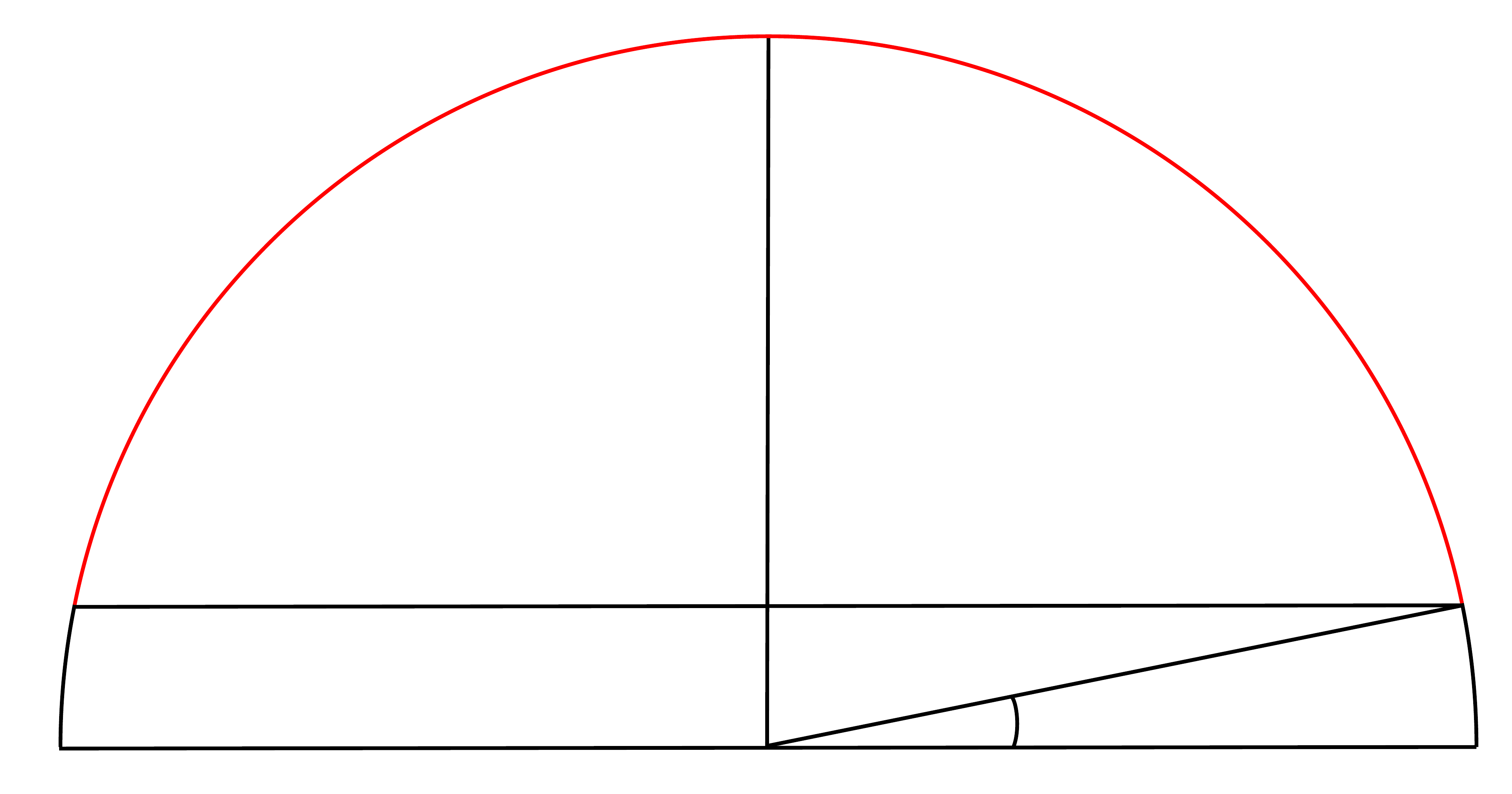}
\caption{A lift of $\beta$ in the upper half-plane.}
\label{PictureLemmaBasmajianLength}
\end{figure}
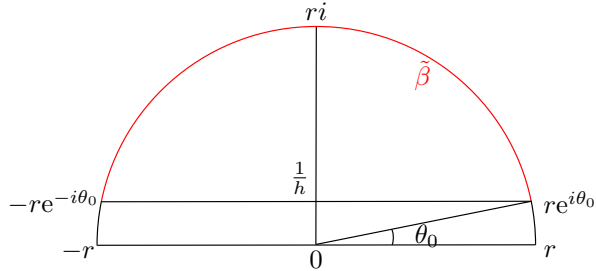    

\noindent
Parametrizing the geodesic segment joining $re^{i\theta_0}$ and $ri$ by $\tilde{\beta} (t) = r {\rm e}^{i t}$ for $\theta_0 \le t \le \frac{\pi}{2}$, a computation yields
\begin{equation}\label{EQLemmaBasmajianExtended}
r = \frac{1}{h} \cosh \frac{\ell(\beta)}{2}.
\end{equation}
Since $\beta$ winds around the cusp $\omega(\beta)$ times, the $f$-translates of $\tilde{\beta}$ intersect $\tilde{\beta}$ exactly $(\omega(\beta)-1)$ times. 
Hence,
$$\omega(\beta)-1 = i(\beta, \beta).$$ 
Thus,
\begin{equation}\label{EQLemmaBasmajianExtendedAdd}
\omega(\beta) = \max \left\{ n \in \mathbb{N}: -r+n \le -r + 2 \sqrt{r^2 - \frac{1}{h^2}} \right\} = \floor*{ 2\sqrt{r^2 - \frac{1}{h^2}} }.
\end{equation}
Plugging $r$ from equation (\ref{EQLemmaBasmajianExtended}) into the expression (\ref{EQLemmaBasmajianExtendedAdd}) allows to conclude after simplification.
\end{proof}

As a direct consequence of the above lemma, we have the following corollary.

\begin{corollary}\label{CorollaryStrandLength}
For $0 < h < \mathcal{h}$, let $\beta \subset \mathcal{B}_{h}$ be a strand that starts on $\partial \mathcal{B}_{h}$, winds $\omega(\beta)$ times around the cusp, and returns to $\partial \mathcal{B}_{h}$. Then 
$$ 2 \arsinh \frac{h (\omega(\beta)-1)}{2} \le \ell(\beta) \le 2 \arsinh \frac{h \omega(\beta)}{2}.$$
\end{corollary}

Using the same computation method as in the previous results, we have the following estimate.

\begin{lemma}\label{LemmaLengthEnoughToEnterALevel}
For $0 < h_0 < h < \mathcal{h}$, let $\beta \subset \mathcal{B}_{h}$ be a strand that has its endpoints on the horocycle of length $h$.
If the length of $\beta$ satisfies  
$$\ell(\beta) > 2 \arcosh \frac{h}{h_0}, $$ 
then $\beta$ must enter $\mathcal{B}_{h_0}$.
\end{lemma}

\begin{proof}
The strand $\beta$ enters $\mathcal{B}_{h_0}$ if and only if it is longer than $\beta_0 \subset \mathcal{B}_{h}$, a strand that has its endpoints on $\partial \mathcal{B}_{h}$, and meets the horocycle of length $h_0$ at the point of tangency. 
By the same computation method as in the proof of Lemma \ref{LemmaBasmajiangeneralization}, the length of $\beta_0$ is given by 
$$ \ell(\beta_0) = 2 \arcosh \frac{h}{h_0},$$
and hence the lemma follows.
\end{proof}

In the following lemma, we estimate the difference between lengths of two strands that have the same winding number but stay in different levels of a cusp neighborhood.

\begin{lemma}\label{LemmaBringUpAStrandToCusp}
Let $\beta_1 \subset \mathcal{B}_{1}$ be a strand with endpoints on the horocycle of length $1$ and
for $0 < h < \frac{1}{10}$, let $\beta_h \subset \mathcal{B}_{h}$ be a strand with endpoints on the horocycle of length $h$.
Suppose that both of them wind around the cusp $\omega$ times.
If $\omega \ge \frac{1}{h}$, then
$$ \ell(\beta_1) > \ell(\beta_h) + 2 \arsinh \frac{\frac{1}{h} - 1}{6}.$$
\end{lemma}

\begin{proof}
\noindent
By Corollary \ref{CorollaryStrandLength}, the lengths of $\beta_1$ and $\beta_h$ satisfy
\begin{equation}\label{EQLemmaBringUpAStrand1}  
\ell(\beta_1) \ge 2 \arsinh \frac{\omega - 1}{2},
\end{equation} 
and
\begin{equation}\label{EQLemmaBringUpAStrand2}  
\ell(\beta_h) \le 2 \arsinh \frac{h \omega}{2}.
\end{equation} 
Define the function 
$$f(x) := \arsinh \frac{x-1}{2} - \arsinh \frac{h x}{2}.$$
Since $f$ is increasing for $x \ge \frac{1}{h}$, we have
$$ f(\omega) \ge f \left( \frac{1}{h} \right).$$
Hence,
\begin{align}
 \arsinh \frac{\omega - 1}{2} - \arsinh \frac{h \omega}{2} 
& \ge \arsinh \frac{\frac{1}{h} - 1}{2} - \arsinh \frac{1}{2} \nonumber \\
& \ge \arsinh \frac{\frac{1}{h} - 1}{6} + \arsinh (1) - \arsinh \frac{1}{2}\label{EQLemmaBringUpAStrand3SumArcsinh} \\
& > \arsinh \frac{\frac{1}{h} - 1}{6},\label{EQLemmaBringUpAStrand3}  
\end{align}
where inequality \eqref{EQLemmaBringUpAStrand3SumArcsinh} holds since $\arsinh(u) + \arsinh(v) < \arsinh(3u v)$ for $u, v \ge 1$.
A combination of inequalities (\ref{EQLemmaBringUpAStrand1}),  (\ref{EQLemmaBringUpAStrand2}) and \eqref{EQLemmaBringUpAStrand3} gives the desired result.
\end{proof}

We now consider sets of strands lying in the same cusp neighborhood.
If $\beta_1, \beta_2 \subset \mathcal{B}_{h}$ are strands with endpoints on the horocycle of length $h$ and with winding numbers $\omega_1 \le \omega_2$, it was shown in \cite[Lemma 3.2]{2016arXiv160900217E} that $i(\beta_1, \beta_2) \le 2 \omega_1$. The following lemma is an extended version of this result.

\begin{lemma}\label{LemmaMultiStrand}
For $h < \mathcal{h}$, let $\mathcal{S} \subset \mathcal{B}_{h}$ be a set consisting of $n \ge 2$ strands $\beta_1, \dots, \beta_n$ with endpoints on the horocycle of length $h$ and winding numbers $\omega_1 \le \dots \le \omega_n$, then 
$$ i(\mathcal{S}, \mathcal{S}) \le (2n-1) \omega_1 + \dots + 3 \omega_{n-1} + \omega_n - n.$$ 
\end{lemma}

\begin{proof}
Using the fact that $i(\beta_i, \beta_i) = \omega_i - 1$ for all $i$ and $i(\beta_i, \beta_j) \le 2\omega_i$ for all $i < j$, we have
\begin{equation*}
\begin{split}
 i(\mathcal{S}, \mathcal{S}) & = i(\beta_1, \beta_1) + \dots + i(\beta_n, \beta_n) + i(\beta_1, \beta_2) + \dots + i(\beta_1, \beta_n) + \dots + i(\beta_{n-1}, \beta_n)\\
& \le \omega_1 + \dots + \omega_n - n + 2 (n-1) \omega_1 + \dots + 2 \omega_{n-1}.
\end{split}
\end{equation*}
Simplifying finishes the proof.
\end{proof}


\section{The proof of Theorem \ref{Theorem1}}\label{SectionProof}

We first define all the constants depending on $X$ that we will need in the proof.
Let $\mathcal{h} = \mathcal{h}(X)$ be the length of the longest embedded horocycle boundary of a cusp of $X$,
and $\mathcal{s} = \mathcal{s}(X)$ be the systole length of $X$.
Set 
$$\mathcal{h_0} = \mathcal{h}_0(X) := \frac{1}{30 \mathcal{h}},$$
and $$\varepsilon_0 = \varepsilon_0(X) := \min \left\{ \frac{\mathcal{h_0}}{5}, \frac{\mathcal{s}}{5} \right\}.$$
Denote by $X_{\varepsilon_0T}$ the $\varepsilon_0$-thick part of $X$, and $X_{\varepsilon_0t}$ the $\varepsilon_0$-thin part of $X$.
We can find the shortest orthogonal distance from each boundary component of $ X_{\varepsilon_0t}$ to itself. 
Let $\mathcal{d}(X)$ be the maximum value of such distances for all boundary components of $X_{\varepsilon_0t}$.
Choose $D(X) \ge 2$ such that
\begin{equation}\label{EQSettingDX}
\frac{\varepsilon_0}{12} \sqrt{x} > 2 \arsinh \frac{\varepsilon_0 x}{2} + \mathcal{d}(X) + 2 \varepsilon_0, \quad \text{ for all } x \ge D(X).
\end{equation}
Set 
\begin{equation*}
\varepsilon = \varepsilon(X) := \frac{\varepsilon_0}{10 D(X)}.
\end{equation*}
Denote by $X_{\varepsilon T}$ the $\varepsilon$-thick part of $X$, and $X_{\varepsilon t}$ the $\varepsilon$-thin part of $X$. 
Let $d(X)$ be the shortest orthogonal distance from the length one horocycle boundary of a cusp to the horocycle itself.
Choose $K(X) \ge D(X)$ such that 
\begin{equation}\label{EQKBig}
C(k,X) := 2 \arsinh(k) + d(X) + 1 < \frac{\varepsilon}{\mathcal{h}} \sqrt[\leftroot{0}\uproot{2}5]{k},  
\quad \text{ for all } k\ge K(X).
\end{equation}
Fix once and for all the constants $\varepsilon, \varepsilon_0, \mathcal{h}_0, D(X)$ and $K(X)$.
For all $k \ge K(X)$, we consider $\mathcal{G}_{\ge k}(X)$, the set of closed geodesics on $X$ that self-intersect at least $k$ times, and investigate those of minimal length.
The following lemma is a key step in our proof.

\begin{lemma}\label{LemmaMain1Proof}
For all $k \ge K(X)$ and $\gamma_k$ a shortest geodesic in $\mathcal{G}_{\ge k}(X)$, $\gamma_k$ enters $X_{\varepsilon_0t}$ exactly once.
\end{lemma}

\begin{proof}
We first note that by inequalities (\ref{EQBasmajianExtended}) and (\ref{EQErlandssonParlier}), $\gamma_k$ must enter $X_{\varepsilon t}$.
Since $\varepsilon < \varepsilon_0$, $\gamma_k$ must reach $X_{\varepsilon_0 t}$ before entering $X_{\varepsilon t}$.
Due to our choice of $\varepsilon_0$, $\gamma_k$ must enter a cusp neighborhood with boundary given by a length one horocycle (such a neighborhood is always embedded). 
Assume that $\gamma_k$ enters $m \ge 1$ such cusp neighborhood(s), say $\mathcal{B}_{1}^1, \dots, \mathcal{B}_{1}^m$.
For $1 \le i \le m$, let $\mathcal{S}_i := \gamma_k \cap \mathcal{B}_{1}^i$. 
Then $\mathcal{S}_i$ is a set consisting of strands starting on $\partial \mathcal{B}_{1}^i$, winding around the respective cusp, and returning to $\partial \mathcal{B}_{1}^i$.
Denote by $\gamma_i^1, \dots, \gamma_i^{n_i}$ the strands in $\mathcal{S}_i$ with winding numbers $\omega_i^{1} \le \dots \le \omega_i^{n_i}$.
Let $\mathcal{S} := \{ \mathcal{S}_1, \dots, \mathcal{S}_m \}$, and $n := n_1 +\dots + n_m \ge 1$.
Recall that at least one strand of $\mathcal{S}$ has to wind multiple times around a cusp to enter $X_{\varepsilon t} \subset X_{\varepsilon_0 t}$.
By our setting, it must cross a horocycle of length $\mathcal{h}_0$.
In what follows, we will show that there is only one such strand.

If $n=1$, then the lemma holds because the strand of $\mathcal{S}$ has to enter $X_{\varepsilon t}$. 
If $n \ge 2$, suppose that $\gamma_1^{n_1}$ and $\gamma_*^{\star}$ are the two strands in $\mathcal{S}$ which have the biggest winding numbers, 
$\omega_1^{n_1}$ and $\omega_*^{\star}$, respectively.
Assume that $\omega_1^{n_1} \ge \omega_*^{\star}$.
If $\omega_*^{\star} < 17 \mathcal{h} $, we claim that $\gamma_*^{\star}$ cannot enter the cusp neighborhood $\mathcal{B}_{\mathcal{h_0}}$ with boundary a horocycle of length $\mathcal{h_0}$. 
Indeed, if $\gamma_*^{\star}$ enters the cusp neighborhood with boundary a horocycle of length $\mathcal{h_0}$, then its length must be at least two times the distance between the horocycles of length $1$ and $\mathcal{h_0}$, which is 
$$2 \log \frac{1}{\mathcal{h_0}} = 2 \log (30 \mathcal{h}),$$
whereas by Corollary \ref{CorollaryStrandLength},
$$ \ell(\gamma_*^{\star}) 
\le 2 \arsinh \frac{\omega_*^{\star}}{2} 
< 2 \arsinh \frac{17 \mathcal{h}}{2}  
< 2 \log (30 \mathcal{h}),$$
which is absurd.
Thus, among all strand of $\mathcal{S}$, only $\gamma_1^{n_1}$ enters a cusp neighborhood with boundary a horocycle of length $\mathcal{h_0}$, and hence the lemma holds.

We now consider the case $\omega_*^{\star} \ge 17 \mathcal{h}$. We first note that due to Lemma \ref{LemmaMultiStrand}, an immediate upper bound on the self-intersection number of $\mathcal{S}$ is given by
\begin{align}
i(\mathcal{S}, \mathcal{S}) & = i(\mathcal{S}_1, \mathcal{S}_1) + \dots + i(\mathcal{S}_m, \mathcal{S}_m)\nonumber\\
& \le (2n_1-1) \omega_1^1 + \dots + \omega_1^{n_1} - n_1 + \dots + (2n_m-1) \omega_m^1 + \dots + \omega_m^{n_m} - n_m. \label{EQLemma1FirstBoundOfIntersectionOfS}
\end{align}
We then proceed as follows.
\begin{itemize}
 \renewcommand{\labelitemi}{\tiny$\blacksquare$}
 \item Evaluating the length of $\mathcal{S}$ by constructing a new set of strands from it.
 \item Bounding the self-intersection number of $\mathcal{S}$ and finalizing the proof.
\end{itemize}

\noindent
\textit{Constructing a new set of strands from $\mathcal{S}$}. 
The construction is an involved cut and paste type argument: by choosing a certain height level on the cusp regions, except for the longest strand $\gamma_1^{n_1}$, we cut all parts of the others which are above this level, and then paste them on top of $\gamma_1^{n_1}$, with the same orientation as the ``winding'' direction of $\gamma_1^{n_1}$. 
More precisely, let $\mathcal{S}^\prime$ be a set consisting of $n$ strands, say $\gamma_i^{j\prime}$ for $1 \le i \le m$ and $1 \le j \le n_i$, such that each $\gamma_i^{j\prime}$ has the same initial and terminal points as $\gamma_i^{j}$, and has winding number $\omega_i^{j\prime}$ that is given by 
$$
\omega_i^{j\prime} := 
  \begin{cases} 
  (2n_1-1) \omega_1^1 + \dots + \omega_1^{n_1} + \dots + (2n_m-1) \omega_m^1 + \dots + \omega_m^{n_m} + 1 & \text{if } (i,j) = (1, n_1), \\
   \omega_i^{j}  & \text{if } \omega_i^{j} < \ceil{3 \mathcal{h}}, \\
   \ceil{3 \mathcal{h}} & \text{otherwise},
  \end{cases}
$$
where the symbol $\ceil{\cdot}$ denotes the ceiling function.
Let $\gamma_k^\prime$ be the geodesic in the free homotopy class of the piecewise geodesic closed curve $(\gamma_k \setminus \mathcal{S}) \cup \mathcal{S}^\prime$. 
We claim that $\gamma_k^\prime \in \mathcal{G}_{\ge k}(X)$.
Indeed, for $\omega_i^{j} \ge \ceil{3 \mathcal{h}}, (i,j) \ne (1, n_1)$, by Corollary \ref{CorollaryStrandLength}, the new strand $\gamma_i^{j \prime}$ satisfies
$$ \ell(\gamma_i^{j\prime}) \ge 
2 \arsinh \frac{\omega_i^{j\prime}-1}{2} 
\ge 2 \arsinh \frac{3 \mathcal{h}-1}{2}
> 2 \arcosh (\mathcal{h}).$$ 
By Lemma \ref{LemmaLengthEnoughToEnterALevel},
with this amount of length, even if it has to stand on the longest embedded horocycle around the cusp, it will still enter $\mathcal{B}_1^{i}$.
Concerning the other strands of $\mathcal{S}^\prime$, they are either longer than the original strands (for $(i,j) = (1, n_1)$) or of the same length (otherwise).
This fact guarantees that
$$ i(\gamma_k^\prime, \gamma_k^\prime) \ge i(\gamma_k, \gamma_k) \ge k.$$
Since $\gamma_k$ realizes the minimal length in $\mathcal{G}_{\ge k}(X)$, we have
$\ell(\gamma_k) \le \ell(\gamma_k^\prime)$.  
As
$$ \ell(\gamma_k)  = \ell(\gamma_k \setminus \mathcal{S}) + \ell(\mathcal{S}), $$
and 
$$ \ell(\gamma_k^\prime) \le \ell(\gamma_k \setminus \mathcal{S} \cup \mathcal{S}^\prime)
= \ell(\gamma_k \setminus \mathcal{S}) + \ell(\mathcal{S}^\prime),$$
we have
\begin{equation}\label{EQCompareSetOfStrandsSAndSPrime}
\ell(\mathcal{S}) \le \ell(\mathcal{S}^\prime).
\end{equation}

\noindent
The length of $\mathcal{S}$ and $\mathcal{S}^\prime$ are as follows:
\begin{equation*}
\begin{split}
\ell(\mathcal{S}) & = \ell(\mathcal{S}_1) + \dots + \ell(\mathcal{S}_m) \\
& = \ell(\gamma_1^1) + \dots + \ell(\gamma_1^{n_1}) + \dots + \ell(\gamma_m^1) + \dots + \ell(\gamma_m^{n_m}),
\end{split}
\end{equation*}
and
\begin{equation*}\label{EQCompareMM3}
\begin{split}
\ell(\mathcal{S}^\prime) & = \ell(\mathcal{S}_1^\prime) + \dots + \ell(\mathcal{S}_m^\prime)  \\
& = \ell(\gamma_1^{1 \prime}) + \dots + \ell(\gamma_1^{n_1 \prime}) + \dots + \ell(\gamma_m^{1 \prime}) + \dots + \ell(\gamma_m^{n_m \prime}).
\end{split}
\end{equation*}
By our construction, the lengths of each pair of strands $\gamma_i^j$ and $\gamma_i^{j\prime}$ satisfy
\begin{equation}\label{EQCompareEachPairOfStrand1} 
\ell(\gamma_i^j) = \ell(\gamma_i^{j\prime}), \quad \text{ for } \omega_i^{j} < \ceil{3 \mathcal{h}},
\end{equation}
and
\begin{equation}\label{EQCompareEachPairOfStrand2} 
\ell(\gamma_i^j) \ge \ell(\gamma_i^{j\prime}), \quad \text{ for } \omega_i^{j} \ge \ceil{3 \mathcal{h}}, (i,j) \ne \{(1,n), (*,\star) \}.
\end{equation}
For $(i,j) = (*,\star)$, recall that $\gamma_*^{\star}$ is the strand in $\mathcal{S}$ that has the second biggest winding number $\omega_*^{\star}$.
Also recall that $\gamma_*^{\star\prime}$ is a strand in $\mathcal{S}^\prime$ that has the same endpoints as $\gamma_*^{\star}$ and winds around the cusp $\ceil{3 \mathcal{h}}$ times. 
Now, let $\gamma_*^{\star\prime\prime}$ be a strand that has its endpoints on the horocycle of length $\frac{1}{15 \mathcal{h}}$ and that winds $\omega_*^{\star\prime\prime}:= \omega_*^{\star}$ times around the cusp (see Figure \ref{PictureStrandGammaStar}).

\begin{figure}[ht]
\labellist
\small\hair 2pt
\pinlabel ${ \color{blue} \gamma_*^{\star\prime}}$ at 620 190
\pinlabel ${ \color{red} \gamma_*^{\star} }$ at 190 500
\pinlabel ${ \color{mygreen} \gamma_*^{\star\prime\prime}}$ at 560 500
\pinlabel ${ \color{blue} \gamma_*^{\star\prime}}$ at 620 190
\pinlabel ${ \color{red} \gamma_*^{\star} }$ at 190 500
\pinlabel ${ \color{mygreen} \gamma_*^{\star\prime\prime}}$ at 560 500

\pinlabel {$\partial \mathcal{B}_1^1$} at -70 050
\pinlabel {$\partial \mathcal{B}_{\frac{1}{15 \mathcal{h}}}^1$} at -40 350
\endlabellist
\centering
\includegraphics[width=0.25\textwidth]{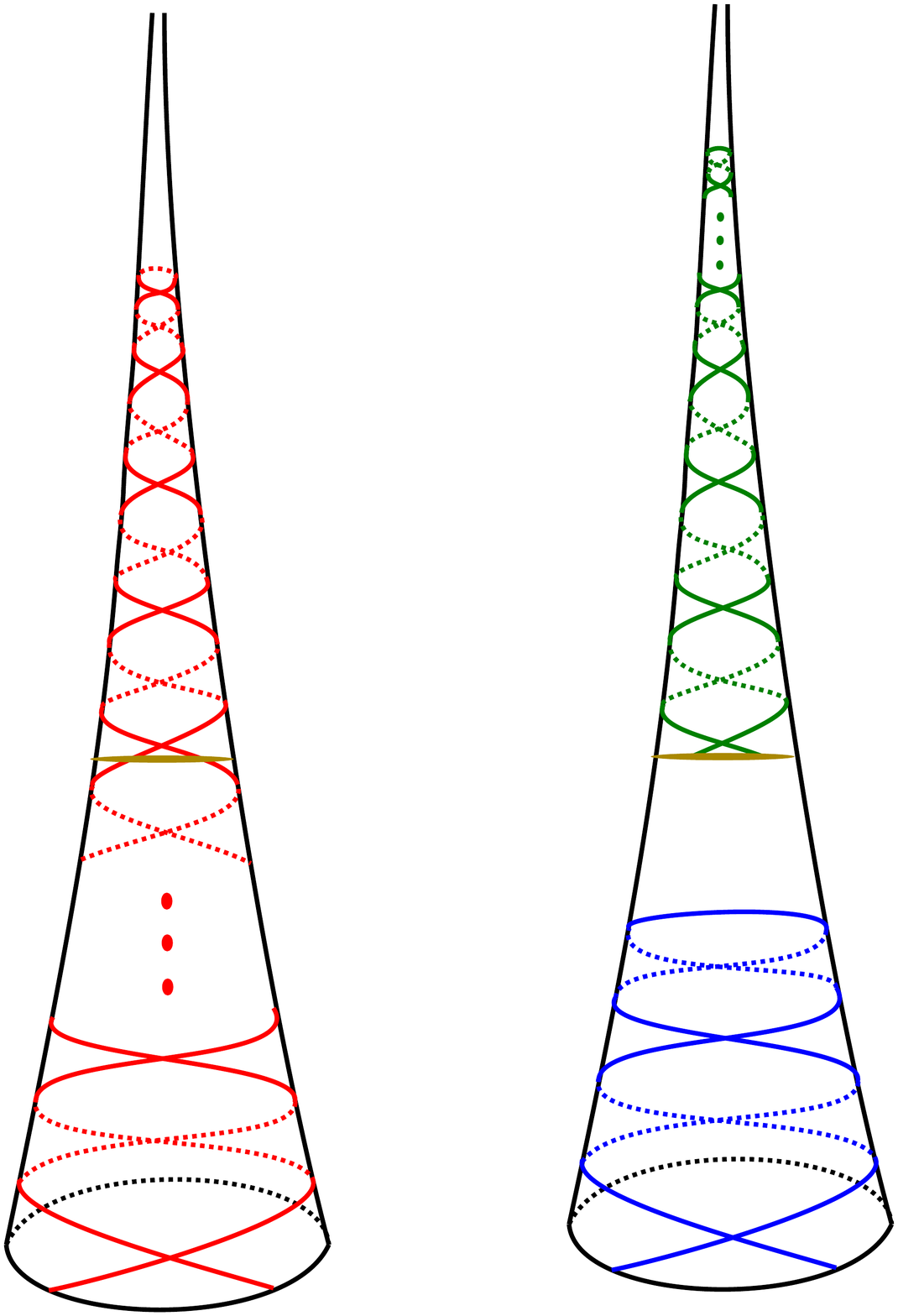}
\caption{Strands $\gamma_*^{\star}, \gamma_*^{\star\prime}$ and $\gamma_*^{\star\prime\prime}$.}
\label{PictureStrandGammaStar}
\end{figure}    

\noindent
By Corollary \ref{CorollaryStrandLength} and Lemma \ref{LemmaBringUpAStrandToCusp}, we have
\begin{align} 
\ell(\gamma_*^{\star}) & \ge \ell(\gamma_*^{\star\prime\prime}) + 2 \arsinh \frac{15 \mathcal{h} - 1}{6}
\ge \ell(\gamma_*^{\star\prime\prime}) + 2 \arsinh \frac{3 \mathcal{h} + 1}{2} \nonumber \\
& \ge \ell(\gamma_*^{\star\prime\prime}) + \ell(\gamma_*^{\star\prime}).\label{EQCompareEachPairOfStrand3} 
\end{align}
From inequalities (\ref{EQCompareEachPairOfStrand1}), (\ref{EQCompareEachPairOfStrand2}) and (\ref{EQCompareEachPairOfStrand3}), we have
\begin{equation*}
\sum_{\substack{1 \le i \le m, 1 \le j \le n_i\\ (i,j) \ne (1,n) }}^{} \ell(\gamma_i^{j})
\ge 
\ell(\gamma_*^{\star\prime\prime}) + \sum_{\substack{1 \le i \le m, 1 \le j \le n_i\\ (i,j) \ne (1,n) }}^{} \ell(\gamma_i^{j\prime}),
\end{equation*}
or equivalently
$$ \ell(\mathcal{S}) -  \ell(\gamma_1^{n_1}) \ge 
\ell(\gamma_*^{\star\prime\prime}) +  \ell(\mathcal{S}^\prime) -  \ell(\gamma_1^{n_1\prime}).$$
With inequality (\ref{EQCompareSetOfStrandsSAndSPrime}), we obtain
\begin{equation}\label{Lemma1GammaNAndNprimeLongest}
\ell(\gamma_1^{n_1 \prime}) \ge \ell(\gamma_1^{n_1}) + \ell(\gamma_*^{\star\prime\prime}).
\end{equation}

\noindent
Using Corollary \ref{CorollaryStrandLength}, we have
\begin{align}
\ell(\gamma_1^{n_1}) + \ell(\gamma_*^{\star\prime\prime})  
& \ge 2 \arsinh \frac{\omega_1^{n_1}-1}{2} + 2 \arsinh \frac{\omega_*^{\star\prime\prime}-1}{30 \mathcal{h}} \nonumber\\
& \ge 2 \arsinh \frac{(\omega_1^{n_1}-1) (\omega_*^{\star\prime\prime}-1)}{30 \mathcal{h}} \nonumber \\
& \ge 2 \arsinh \frac{(\omega_1^{n_1}-1) (\omega_*^{\star}-1)}{30 \mathcal{h}}, \label{Lemma1GammaNLongest}
\end{align}
and
$$\ell(\gamma_1^{n_1 \prime}) \le 2 \arsinh \frac{\omega_1^{n_1\prime}}{2}.$$
Recall that
$$\omega_1^{n_1 \prime} = (2n_1-1) \omega_1^1 + \dots + \omega_1^{n_1} + \dots + (2n_m-1) \omega_m^1 + \dots + \omega_m^{n_m} + 1,$$
and that $\omega_1^{n_1}$ and $\omega_*^{\star}$ are the two biggest winding numbers of the strands in $\mathcal{S}$. 
We then have
\begin{equation*}
\begin{split}
\omega_1^{n_1\prime}
& \le \omega_1^{n_1} + (n_1^2-1)\omega_*^{\star} + \dots + n_m^2 \omega_*^{\star} + 1\\
& < \omega_1^{n_1} + (n_1^2 + n_2^2 + \dots + n_m^2 - 1) \omega_*^{\star} +  \omega_*^{\star}\\
& \le \omega_1^{n_1} + n^2 \omega_*^{\star}.
\end{split}
\end{equation*}
Therefore,
\begin{equation}\label{Lemma1GammaNprimeLongest}
\ell(\gamma_1^{n_1^\prime}) < 2 \arsinh \frac{\omega_1^{n_1} + n^2 \omega_*^{\star}}{2}.
\end{equation}
From inequalities (\ref{Lemma1GammaNLongest}) and (\ref{Lemma1GammaNprimeLongest}), we can rewrite inequality (\ref{Lemma1GammaNAndNprimeLongest}) in terms of winding numbers as follows:
$$2 \arsinh \frac{\omega_1^{n_1} + n^2 \omega_*^{\star}}{2} 
> 2 \arsinh \frac{(\omega_1^{n_1}-1) (\omega_*^{\star}-1)}{30 \mathcal{h}}.$$
Simplifying and rearranging yields
$$(15 \mathcal{h} n^2 + 1) \omega_*^{\star} - 1 > 
(\omega_*^{\star} - 15 \mathcal{h} - 1) \omega_1^{n_1}.$$
Since $\omega_*^{\star} \ge 17 \mathcal{h}$, we have

\begin{align}
\omega_1^{n_1} 
& < \frac{(15 \mathcal{h} n^2 + 1) \omega_*^{\star} - 1}{\omega_*^{\star} - 15 \mathcal{h} - 1} 
< \frac{16 \mathcal{h} n^2 \omega_*^{\star}}{\omega_*^{\star} - 16 \mathcal{h}} 
= 16 \mathcal{h} n^2 + \frac{16^2 \mathcal{h}^2 n^2}{\omega_*^{\star} - 16 \mathcal{h}} \nonumber \\
& \le  16 \mathcal{h} n^2 + \frac{16^2 \mathcal{h}^2 n^2}{\mathcal{h}} 
= 272 \mathcal{h} n^2. \label{EQBiggestWindingNumberAndTheNumberOfStrand}
\end{align}

\noindent
\textit{Bounding the self-intersection number of $\mathcal{S}$ and finalizing the proof.}
Notice that $\gamma_k$ contains $n$ geodesic segments entering cusp neighborhoods with boundary horocycles of length one. 
Each of them has to pass through an open cylinder bounded by the horocycle of length $2$ and $1$, and then return (see Figure \ref{PictureOpenCylinderOneTwo}).
Such a cylinder is always embedded in $X$, and moreover, if $\mathcal{C}$ and $\mathcal{C}^\prime$ are distinct cusps, then the associated cylinders are disjoint. 
Therefore,
$$\ell(\gamma_k) > n \times 2 \log (2) > n.$$ 

\begin{figure}[ht]
\labellist
\small\hair 2pt
\pinlabel {$\log(2)$} at 320 250
\pinlabel {a horocycle of length $1$} at -85 540
\pinlabel {a horocycle of length $2$} at -280 50
\endlabellist
\centering
\includegraphics[width=0.28\textwidth]{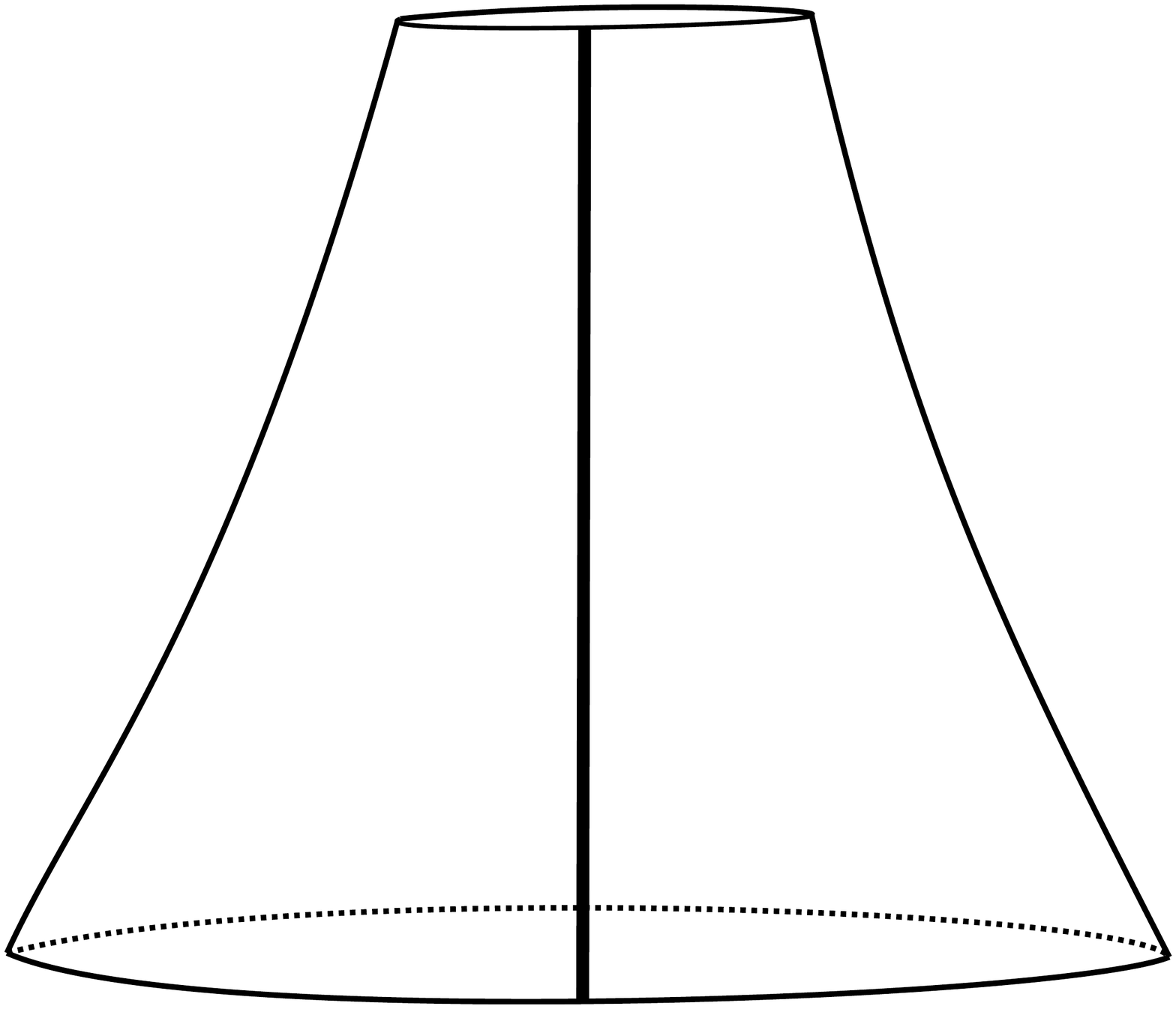}
\caption{A cylinder bounded by the horocycles of length $1$ and $2$.}
\label{PictureOpenCylinderOneTwo}
\end{figure}    

\noindent
Since $\ell(\gamma_k) \le C(k,X)$, it follows that 
\begin{equation}\label{EQBoundingNLemma1}
n < C(k,X).
\end{equation}
Putting together inequalities (\ref{EQKBig}), (\ref{EQLemma1FirstBoundOfIntersectionOfS}),  (\ref{EQBiggestWindingNumberAndTheNumberOfStrand}) and (\ref{EQBoundingNLemma1}), we   have
\begin{align}
i(\mathcal{S}, \mathcal{S}) & \le 
(2n_1-1) \omega_1^1 + \dots + \omega_1^{n_1} - n_1 + \dots + (2n_m-1) \omega_m^1 + \dots + \omega_m^{n_m} - n_m \nonumber \\
& < (n_1^2 + \dots + n_m^2) \omega_1^{n_1} 
<  n^2 \omega_1^{n_1} 
< 272 \mathcal{h} n^4 \le 272 \mathcal{h} (C(k,X))^4 \nonumber\\
& < 272 \mathcal{h} \left(\frac{\varepsilon}{\mathcal{h}} \sqrt[\leftroot{0}\uproot{2}5]{k} \right)^4 
< \frac{k}{10}. \label{Lemma1UpperBoundForIntersectionsOfS}
\end{align}

\noindent
On the other hand, the $\varepsilon_0$-thick-thin decomposition of $X$ decomposes $\gamma_k$ into two parts,
$\gamma_{\varepsilon_0 T} := \gamma_k \cap X_{\varepsilon_0 T}$ and 
$\gamma_{\varepsilon_0 t} := \gamma_k \cap X_{\varepsilon_0 t}$.
By inequality (\ref{EQErlandssonParlier}), the self-intersection number of $\gamma_{\varepsilon_0 T}$ satisfies
$$ \ell(\gamma_{\varepsilon_0 T}) > \frac{\varepsilon_0}{12} \sqrt{i(\gamma_{\varepsilon_0 T}, \gamma_{\varepsilon_0 T})}.$$
Using the fact that $\ell(\gamma_{\varepsilon_0 T}) < \ell(\gamma_k) \le C(k,X)$, and inequality (\ref{EQKBig}), we have
\begin{align}
i(\gamma_{\varepsilon_0 T}, \gamma_{\varepsilon_0 T}) 
& < \left(\frac{12}{\varepsilon_0}\right)^2 (\ell(\gamma_k))^2 
< \left(\frac{12}{\varepsilon_0}\right)^2 (C(k,X))^2 \nonumber\\
& \le \left(\frac{12}{\varepsilon_0}\right)^2 \left(\frac{\varepsilon}{\mathcal{h}} \sqrt[\leftroot{0}\uproot{2}5]{k} \right)^2 
< \frac{k}{10}. \label{EQProofTheFirstPartIntersectThick}
\end{align}

\noindent
Note that by the choice of $\varepsilon_0$, the self-intersection number that $\gamma_k$ makes in $\cup_{1\le i \le m} \mathcal{B}_{1}^i$ must be at least the number that $\gamma_k$ makes in the closure of  $X_{\varepsilon_0 t}$ (which means that we also count the self-intersection points of $\gamma_k$ lying in $\partial X_{\varepsilon_0 t}$). We therefore have
$$ i(\mathcal{S}, \mathcal{S}) 
\ge i(\gamma_k, \gamma_k) - i(\gamma_{\varepsilon_0 T}, \gamma_{\varepsilon_0 T}) 
> k - \frac{k}{10} = \frac{9k}{10},$$
which contradicts to inequality (\ref{Lemma1UpperBoundForIntersectionsOfS}).
Thus, $\omega_*^{\star} < 17 \mathcal{h}$, and we are done.
\end{proof}

We now can prove the first part of Theorem \ref{Theorem1}.

\begin{proof}[Proof of Theorem \ref{Theorem1Part1}] 
By Lemma \ref{LemmaMain1Proof}, for all $k \ge K(X)$ and $\gamma_k$ a shortest geodesic in $\mathcal{G}_{\ge k}(X)$, $\gamma_k$ enters the $\varepsilon_0$-thin part of $X$ only once.
Let $\mathcal{B} \subset X_{\varepsilon_0t}$ be the only cusp neighborhood with boundary a component of $\partial X_{\varepsilon_0t}$ that $\gamma_k$ enters.
We can decompose $\gamma_k$ into two parts, $\gamma_{\varepsilon_0t} = X_{\varepsilon_0t} \cap \gamma_k $ 
and $\gamma_{\varepsilon_0T} = X_{\varepsilon_0T} \cap \gamma_k$.
We know that $\gamma_{\varepsilon_0t}$ is a strand, so $\gamma_{\varepsilon_0T}$ is also a continuous image of an interval.
By inequality (\ref{EQProofTheFirstPartIntersectThick}),
the self-intersection number of $\gamma_{\varepsilon_0T}$ is strictly smaller than $k$.
Now, suppose that we are at the initial point of $\gamma_{\varepsilon_0T}$. By following $\gamma_{\varepsilon_0T}$, we reach its terminal point, which is in $\partial \mathcal{B}$.  
The only way we can continue is to go into the cusp neighborhood $\mathcal{B}$ by following a strand. 
Since every loop we wrap around the cusp gives us one extra self-intersection, we will stop at the moment when we get enough $k$ self-intersection points. Thus, $i(\gamma_k, \gamma_k) = k$.
\end{proof}

\begin{remark}\label{remarkKXdirectly}
We can find the constant $K(X)$ directly, meaning not via $D(X)$, which implies that $K(X)$ can be made much smaller, such that for all $k \ge K(X)$, if $\gamma_k \in \mathcal{G}_{\ge k}(X)$ is a minimal length geodesic, then $\gamma_k$ self-intersects $k$ times (see the proof of Theorem \ref{Theorem1Part1} above). 
However, by doing so, we do not know what $\gamma_k$ looks like.
\end{remark}

We now deal with the second part of Theorem \ref{Theorem1}.
Since we want to be able to determine what all geodesics of minimal length in $ \mathcal{G}_{\ge k}(X)$ look like for $k \ge K(X)$, we turn our attention to $X_{\varepsilon_0T}$, the $\varepsilon_0$-thick part of $X$.

\begin{lemma}\label{LemmaMain2Proof} 
For all $k \ge K(X)$ and $\gamma_k$ a shortest geodesic in $\mathcal{G}_{\ge k}(X)$,
the self-intersection of $\gamma_{\varepsilon_0T} = X_{\varepsilon_0T} \cap \gamma_k$ satisfies
$$i(\gamma_{\varepsilon_0T}, \gamma_{\varepsilon_0T}) < D(X).$$
\end{lemma}

\noindent
Recall that $D(X)$ is defined in (\ref{EQSettingDX}). It is chosen in such a way that
$$ \frac{\varepsilon_0}{12} \sqrt{x} > 2 \arsinh \frac{\varepsilon_0 x}{2} + \mathcal{d}(X) + 2 \varepsilon_0, \quad \text{ for all } x \ge D(X),$$
where $\mathcal{d}(X)$ is the maximum value of all shortest orthogonal distances from each boundary component of $X_{\varepsilon_0t}$ to itself.

\begin{proof}[Proof of Lemma \ref{LemmaMain2Proof}]
We construct a geodesic $\sigma \in \mathcal{G}_{\ge k}(X)$, which by definition of $\gamma_k$, must not be shorter than $\gamma_k$.
Our construction is based on the procedure used in \cite[Proposition 4.2]{MR3065183}.
Let $\mathcal{B} \subset X_{\varepsilon_0t}$ be as in the proof of Theorem \ref{Theorem1Part1}, the only cusp neighborhood with boundary a component of $\partial X_{\varepsilon_0t}$ that $\gamma_k$ enters.
Let $p, q \in \partial \mathcal{B}$ be the initial and terminal points of $\gamma_{\varepsilon_0t}$, respectively.
Let $\delta_1$ be the shortest orthogonal geodesic from $\partial \mathcal{B}$ to $\partial \mathcal{B}$ itself.
In our setting, $\delta_1$ has length at most $\mathcal{d}(X)$, see the beginning of Section \ref{SectionProof}. 
After putting an orientation on $\delta_1$, denote its initial and terminal points by $x$ and $y$, respectively.
Let $\delta_2$ be the shortest segment along $\partial \mathcal{B}$ from $q$ to $x$, and 
$\delta_3$ be the shortest segment along $\partial \mathcal{B}$ from $y$ to $p$. Each of $\delta_2$ and $\delta_3$ has length at most $\varepsilon_0$.
Let $a$ be the point of $\gamma_{\varepsilon_0t}$ furthest away from $\partial \mathcal{B}$, and $\alpha$ be the strand that has its endpoints at $a$ and winds around the cusp $i(\gamma_{\varepsilon_0T}, \gamma_{\varepsilon_0T})$ times. 
We orient $\alpha$ in the ``winding'' direction of $\gamma_{\varepsilon_0t}$. 
Now, define $\sigma$ to be the closed geodesic in the free homotopy class of the curve obtained by following $\gamma_{\varepsilon_0t}$ from $p$ to $a$, then following the strand $\alpha$, then continuing along $\gamma_{\varepsilon_0t}$ to $q$, and finally following $\delta_2 * \delta_1 * \delta_3$ to come back to $p$ (see Figure \ref{PictureLemma2Sigma}). 
\begin{figure}[ht]
\labellist
\small\hair 2pt
\pinlabel ${\color{red} \alpha}$ at 340 740
\pinlabel ${\color{blue} \gamma_{\varepsilon_0t}}$ at 380 570
\pinlabel ${\color{mygreen} \delta_1}$ at 260 170
\pinlabel {${\color{heartgold}\xuparrow{1.12cm}}$ a region in the $X_{\varepsilon t}$} at 590 740
\pinlabel {the horocycle $h_a$} at 120 700
\pinlabel {$\partial \mathcal{B}$} at 180 280
\pinlabel ${\color{kellygreen} a}$ at 340 700
\pinlabel {\color{blue} $p$} at 235 234
\pinlabel {\color{blue} $q$} at 340 234
\pinlabel {\color{mygreen} $x$} at 310 305
\pinlabel {\color{mygreen} $y$} at 300 235
\pinlabel ${\color{deepmagenta} \delta_3}$ at 260 230
\pinlabel ${\color{deepcarrotorange} \delta_2}$ at 399 280
\pinlabel ${\color{kellygreen} a}$ at 340 700
\pinlabel ${\color{kellygreen} a}$ at 340 700
\pinlabel {\color{blue} $p$} at 235 234
\pinlabel {\color{blue} $q$} at 340 234
\pinlabel {\color{mygreen} $x$} at 310 305
\pinlabel {\color{mygreen} $y$} at 300 235
\pinlabel ${\color{deepmagenta} \delta_3}$ at 260 230
\pinlabel ${\color{deepcarrotorange} \delta_2}$ at 399 280
\pinlabel ${\color{red} \alpha}$ at 340 740
\pinlabel ${\color{blue} \gamma_{\varepsilon_0t}}$ at 380 570
\pinlabel ${\color{mygreen} \delta_1}$ at 260 170
\endlabellist
\centering
\includegraphics[width=0.38\textwidth]{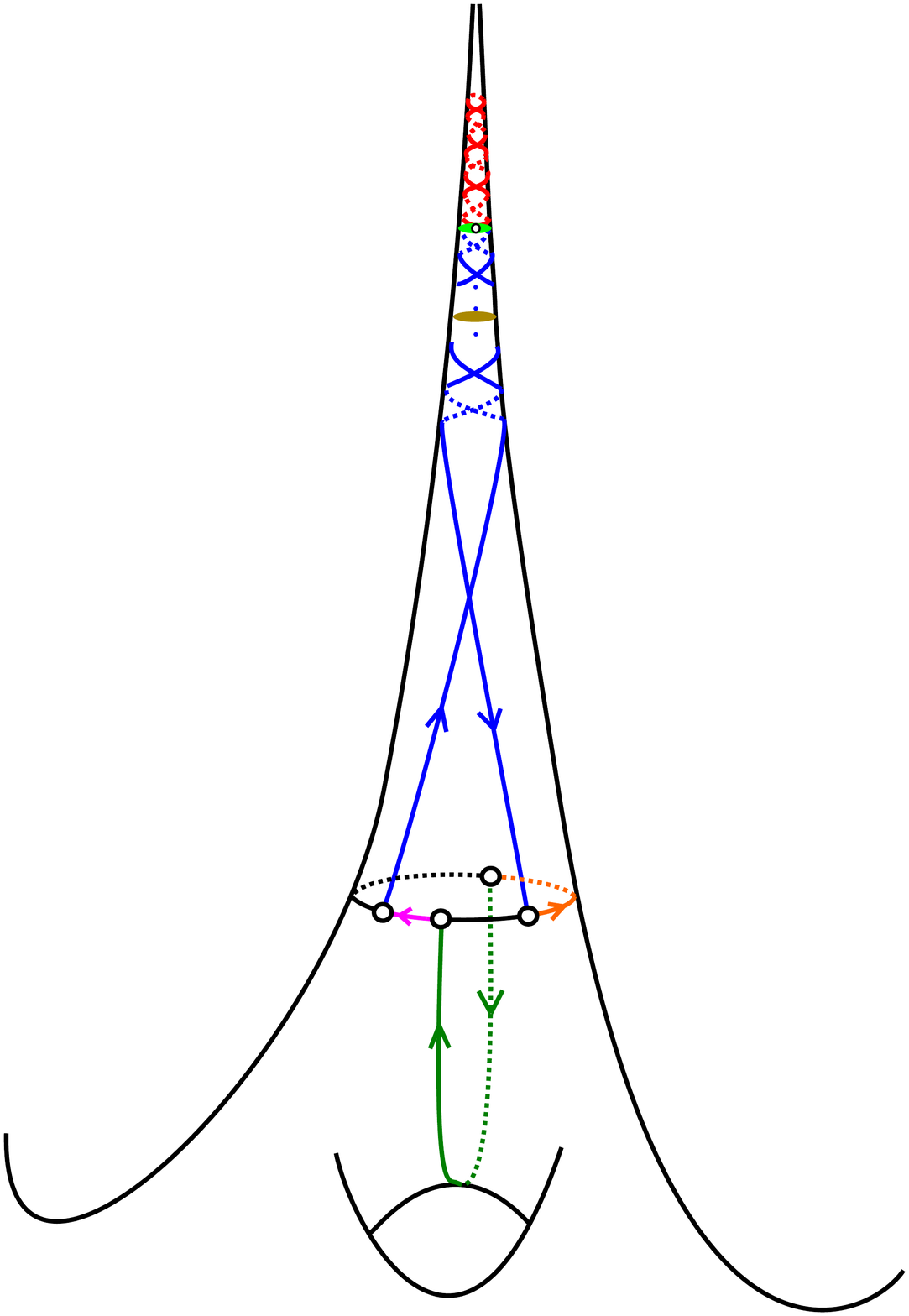}
\caption{Constructing the geodesic $\sigma \in \mathcal{G}_{\ge k}(X)$.}
\label{PictureLemma2Sigma}
\end{figure}    
By the construction, $\sigma \in \mathcal{G}_{\ge k}(X)$, and hence 
$\ell(\gamma_k) \le \ell(\sigma)$.
We have
$$\ell(\sigma) \le \ell(\gamma_{\varepsilon_0t}) + \ell(\alpha) + \ell(\delta_1) + \ell(\delta_2) + \ell(\delta_3)
\le \ell(\gamma_{\varepsilon_0t}) + \ell(\alpha) + \mathcal{d}(X) + 2 \varepsilon_0.$$
Let $h_a$ be the horocycle in the cusp neighborhood $\mathcal{B}$ passing through $a$.
By abuse of notation, we denote the length of this horocycle also by $h_a$.
Note that by our setting, $h_a \le 2 \varepsilon < \varepsilon_0$.
Using Corollary \ref{CorollaryStrandLength}, we have
$$ \ell(\alpha) \le 2 \arsinh \frac{h_a i(\gamma_{\varepsilon_0T}, \gamma_{\varepsilon_0T})}{2}
< 2 \arsinh \frac{\varepsilon_0 i(\gamma_{\varepsilon_0T}, \gamma_{\varepsilon_0T})}{2}.$$
Therefore, 
$$\ell(\sigma) < \ell(\gamma_{\varepsilon_0t}) + 
2 \arsinh \frac{\varepsilon_0 i(\gamma_{\varepsilon_0T}, \gamma_{\varepsilon_0T})}{2} + \mathcal{d}(X) + 2 \varepsilon_0.$$
As $\ell(\gamma_k) \le \ell(\sigma)$, the above inequality also holds for $\gamma_k$.
By inequality (\ref{EQErlandssonParlier}), we have
$$\ell(\gamma_k) = \ell(\gamma_{\varepsilon_0t}) + \ell(\gamma_{\varepsilon_0T})
> \ell(\gamma_{\varepsilon_0t}) + \frac{\varepsilon_0}{12} \sqrt{i(\gamma_{\varepsilon_0T}, \gamma_{\varepsilon_0T})}.$$
Thus,
$$\frac{\varepsilon_0}{12} \sqrt{i(\gamma_{\varepsilon_0T}, \gamma_{\varepsilon_0T})} 
< 2 \arsinh \frac{\varepsilon_0 i(\gamma_{\varepsilon_0T}, \gamma_{\varepsilon_0T})}{2} + \mathcal{d}(X) + 2 \varepsilon_0.$$
By the choice of $D(X)$, see \eqref{EQSettingDX}, the above inequality does not hold if $i(\gamma_{\varepsilon_0T}, \gamma_{\varepsilon_0T}) \ge D(X)$. 
Hence, $i(\gamma_{\varepsilon_0T}, \gamma_{\varepsilon_0T}) < D(X)$ as desired.
\end{proof}

We finally can complete the proof of Theorem \ref{Theorem1}.

\begin{proof}[Proof of Theorem \ref{Theorem1Part2}]
By Lemma \ref{LemmaMain1Proof}, the geodesic $\gamma_k$ can be decomposed into $\gamma_{\varepsilon_0t}$ and $\gamma_{\varepsilon_0T}$, which are both continuous images of intervals. 
By Lemma \ref{LemmaMain2Proof}, the self-intersection number of $\gamma_{\varepsilon_0T}$ is strictly bounded from above by $D(X)$.
Let $u \in X_{\varepsilon_0t}$ be the self-intersection of $\gamma_{\varepsilon_0t}$ that is nearest to $\partial \mathcal{B}$. 
Let $\mathcal{u}$ be the closed curve obtained by starting from $u$, following $\gamma_{\varepsilon_0t}$ and returning to $u$. 
Then $\mathcal{u}$ is freely homotopic to a power of $\tau$, where $\tau$ is a loop based at $u$ and freely homotopic to the only cusp that $\gamma_{\varepsilon_0t}$ winds around.
Define $\kappa := \gamma_k \setminus \mathcal{u}$.
We first note that the closed curve $\kappa$ is not a power of $\tau$ because $\gamma_k$ is primitive. 
Besides, $\kappa$ has the same self-intersection number as  $\gamma_{\varepsilon_0T}$, which is less than $D(X)$.

If $i(\gamma_{\varepsilon_0T},\gamma_{\varepsilon_0T}) = i(\kappa, \kappa) = 0$, then 
$\gamma_k$ is freely homotopic to $\kappa * \mathcal{u} = \kappa * \tau^k$, where in this case, $\kappa$ is a non-contractible simple loop based at $u$ and not freely homotopic to the cusp. Thus, the theorem holds.

We now consider the case $i(\gamma_{\varepsilon_0T},\gamma_{\varepsilon_0T}) = i(\kappa, \kappa) \ge 1$. 
By starting from its initial point, $\gamma_{\varepsilon_0T}$ travels inside $X_{\varepsilon_0T}$ until it crosses itself for the first time. At this moment, it makes the first loop $\mathcal{v}$ and the first self-intersection point $v$ (see Figure \ref{FigurePictureBigon}). 
More precisely, we choose a parametrization 
$\gamma_{\varepsilon_0T}: [0, 1] \to X$ 
and let $t_2$ be the supremum of all $t$ such that the restriction 
$\gamma_{\varepsilon_0T} | _{[ 0, t ]}$ is a simple arc.
Then there exists a unique $t_1 \in [0, t_2]$ with 
$\gamma_{\varepsilon_0T} (t_1) = \gamma_{\varepsilon_0T}(t_2)$, and it follows that
$\gamma_{\varepsilon_0T} | _{[t_1, t_2]}$ is a simple loop in $\gamma_{\varepsilon_0T}$. 
This loop is denoted by $\mathcal{v}$ and the point $\gamma_{\varepsilon_0T} (t_1)$ is denoted by $v$. 
The simple loop $\mathcal{v}$ is non-contractible, so we consider whether it is freely homotopic to the cusp or not.

If $\mathcal{v}$ is freely homotopic to the cusp, 
we note that $\gamma_{\varepsilon_0T}$ must go out of the cusp neighborhood with boundary $\mathcal{v}$, 
because otherwise, $\kappa$ would only loop around the cusp, which contradicts the fact that $\gamma_k$ is primitive.
To get out of the cusp neighborhood, $\gamma_{\varepsilon_0T}$ must create a bigon, which means that $\gamma_k$ has \textit{excess} self-intersections (see \cite{MR804478}), and this is absurd.
Hence, $\mathcal{v}$ is not freely homotopic to the cusp.

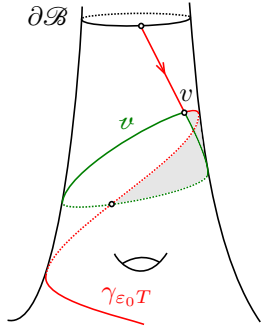
\begin{figure}
\labellist
\small\hair 2pt
\pinlabel ${\partial \mathcal{B}}$ at 80 560
\pinlabel ${v}$ at 325 420
\pinlabel ${\color{mygreen}\mathcal{v}}$ at 220 370
\pinlabel ${\color{red}\gamma_{\varepsilon_0T}}$ at 220 60
\endlabellist
\includegraphics[width=0.22\textwidth]{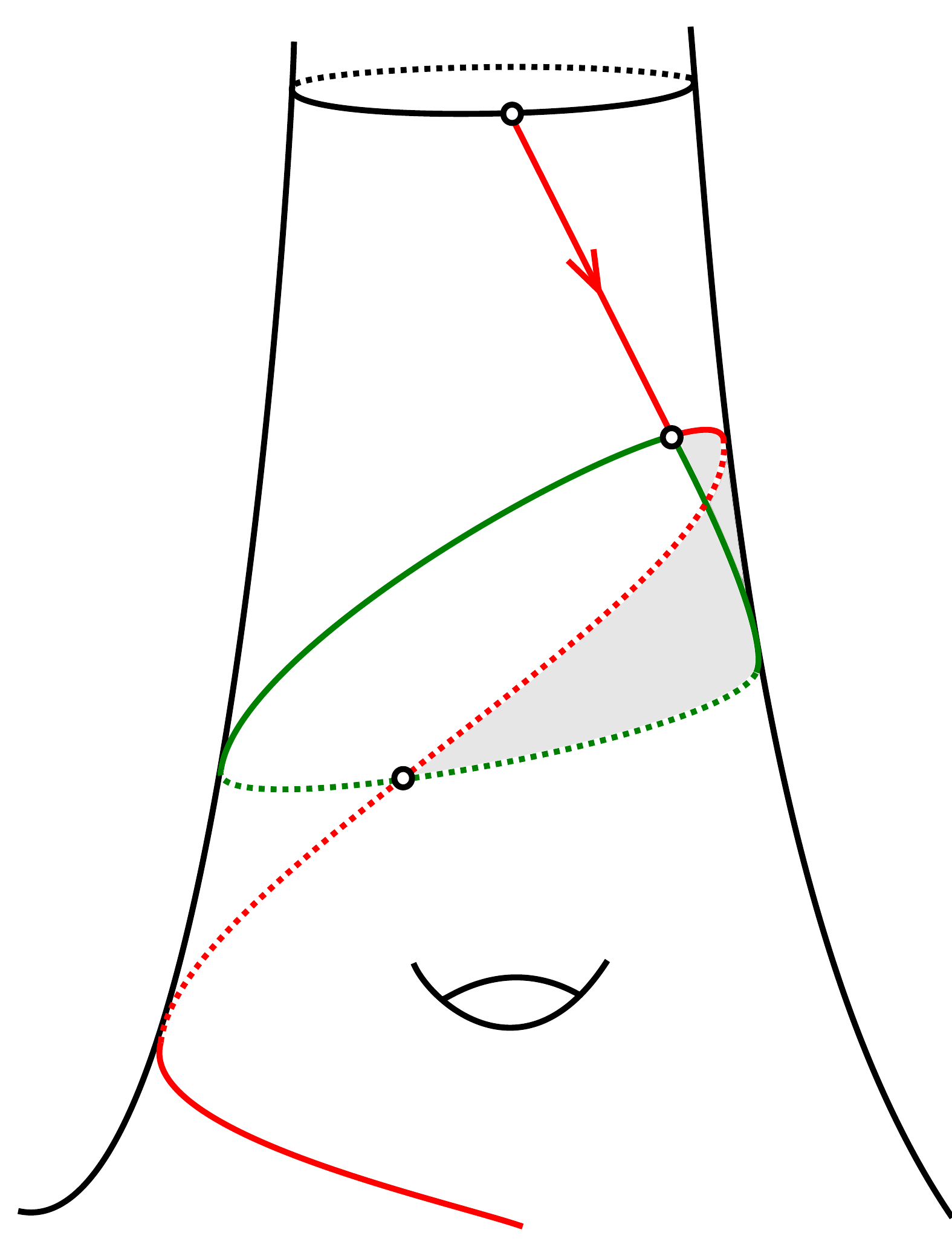}
\caption{If $\mathcal{v}$ is freely homotopic to the cusp, then $\gamma_{\varepsilon_0T}$ will create a bigon.}
\label{FigurePictureBigon}
\end{figure}

Define $\chi := \kappa \setminus \mathcal{v}$. 
If $\chi$ is freely homotopic to a power of $\tau$, then $\gamma_k$ is freely homotopic to $\mathcal{v} * \tau^k$, and hence the theorem holds.

Suppose that $\chi$ is not freely homotopic to a power of $\tau$.
By observing that the loop $\mathcal{v} \subset X_{\varepsilon_0T}$ has length at least $2 \varepsilon_0$,
the rough idea is now the following: we have an extra $2 \varepsilon_0$ worth of length which we can use to wrap around the cusp, and since our $\varepsilon$ is chosen to be small enough, this adds as many self-intersections as we want.

Let $p, q, a, \alpha, h_a$ be as in the proof of Lemma \ref{LemmaMain2Proof}, which are as follows:
$p, q \in \partial \mathcal{B}$ are the initial and terminal points of $\gamma_{\varepsilon_0t}$, respectively,
$a$ is the point of $\gamma_{\varepsilon_0t}$ furthest away from $\partial \mathcal{B}$, 
$\alpha$ is the strand that has its endpoints at $a$ and winds around the cusp $i(\gamma_{\varepsilon_0T}, \gamma_{\varepsilon_0T})$ times (in the same direction as the  ``winding'' direction of $\gamma_{\varepsilon_0t}$),
$h_a$ is the horocycle passing through $a$ in the cusp neighborhood $\mathcal{B}$, 
and by abuse of notation, we also denote by $h_a$ the length of this horocycle.

Let $\sigma$ be the closed geodesic in the free homotopy class of $\alpha * \mathcal{u} * \chi$ 
(the curve obtained by following $\gamma_{\varepsilon_0t}$ from $p$ to $a$, then following the strand $\alpha$, then continuing along $\gamma_{\varepsilon_0t}$ to $q$, 
then following $\gamma_{\varepsilon_0T}$ until $v$, skipping the loop $\mathcal{v}$ to return immediately to $p$ by following $\gamma_{\varepsilon_0T}$).
By our construction, $\sigma \in \mathcal{G}_{\ge k}(X)$, and hence 
$\ell(\gamma_k) \le \ell(\sigma)$.
Using Corollary \ref{CorollaryStrandLength} and Lemma \ref{LemmaMain2Proof}, we have
\begin{align*}
\ell(\gamma_k) & \le \ell(\sigma) \le \ell(\chi) + \ell(\mathcal{u}) + \ell(\alpha)  = \ell(\gamma_{\varepsilon_0T}) + \ell(\gamma_{\varepsilon_0t}) - \ell(\mathcal{v}) + \ell(\alpha)\\
& \le \ell (\gamma_k) - 2 \varepsilon_0 + 2 \arsinh \frac{h_a i(\gamma_{\varepsilon_0T}, \gamma_{\varepsilon_0T})}{2} \\ 
& \le \ell(\gamma_{k}) - 2 \varepsilon_0 + 2 \arsinh \frac{2 \varepsilon D(X)}{2} \\
& \le \ell(\gamma_{k}) - 2 \varepsilon_0 + 2 \arsinh \left( \frac{\varepsilon_0}{10 D(X)} D(X) \right) \\
& \le \ell(\gamma_{k}) - 2 \varepsilon_0 + 2 \arsinh \left( \frac{\varepsilon_0}{10} \right) \\
& < \ell(\gamma_{k}),
\end{align*}
which is a contradiction.

In conclusion, $\gamma_k$ is freely homotopic to $\kappa * \tau^k$, where $\tau$ is freely homotopic to the cusp, and $\kappa$ is a non-contractible simple loop that is not freely homotopic to $\tau$.
\end{proof}

\section{Corollary and remarks}\label{SectionCorollary}

In this section, we will consider $\mathcal{M}_{g,n}^{\mathcal{s}}$, the set of all cusped hyperbolic surfaces of genus $g$ with $n$ punctures or boundary components, that are not the thrice punctured sphere\footnote{The thrice punctured sphere will be treated in Remark \ref{RemarkSectionCorollaryThricePuncturedSphere}.}, and the systole lengths of which are at least $\mathcal{s}$, where $\mathcal{s} > 0$. 
In order to prove Corollary \ref{Theorem2}, we will make all the constants defined at the beginning of Section \ref{SectionProof} become constants that depend on $g,n$ and $\mathcal{s}$.
But before doing so, we first perform some computations related to pairs of pants with cusp(s). 

\begin{lemma}\label{Section4LemmaOrthogonalDistance}
Let $Y$ be a pair of pants with at least one cusp.
Let $\mathcal{d}$ be the shortest orthogonal distance from the horocycle of length $h$ to itself.

\begin{enumerate}
  \item If $Y$ has one cusp and two boundary geodesics of length $\alpha$ and $\beta$, then 
  $$\mathcal{d} = \mathcal{d}(h, \alpha, \beta) = 2 \log \frac{2 (\cosh \frac{\alpha}{2} + \cosh \frac{\beta}{2} ) }{h}.$$ 
  \item If $Y$ has two cusps and a boundary geodesic of length $\alpha$, then
$$\mathcal{d} = \mathcal{d}(h, \alpha) = 2 \log \frac{2 (\cosh \frac{\alpha}{2} + 1)}{h}.$$
\item If $Y$ is the thrice punctured sphere, then
$$\mathcal{d} = \mathcal{d}(h) = 2 \log \frac{4}{h}.$$
\end{enumerate}

\end{lemma}

\begin{figure}[ht]
\labellist
\small\hair 2pt
\pinlabel {$h$} at 250 480
\pinlabel ${\mathcal{d}}$ at 260 150
\pinlabel ${\alpha}$ at 80 120
\pinlabel ${\beta}$ at 560 120
\pinlabel ${p_{\alpha\beta}}$ at 350 0
\pinlabel ${\frac{\tilde{\alpha}}{2}}$ at 1450 235
\pinlabel ${\frac{\tilde{\beta}}{2}}$ at 1919 170
\pinlabel ${\frac{\mathcal{d}}{2}}$ at 1720 350
\pinlabel ${\tilde{p}_{\alpha\beta}}$ at 1600 250

\pinlabel ${0}$ at 1690 -20
\pinlabel ${-1}$ at 1450 -20
\pinlabel ${1}$ at 1910 -20
\pinlabel ${x_\alpha}$ at 1240 -20
\pinlabel ${y_\alpha}$ at 1570 -20
\pinlabel ${x_\beta}$ at 1860 -20
\pinlabel ${y_\beta}$ at 2100 -20

\pinlabel {$y = \frac{x_\beta + y_\beta - x_\alpha - y_\alpha}{h}$} at 850 480
\endlabellist
\includegraphics[width=0.82\textwidth]{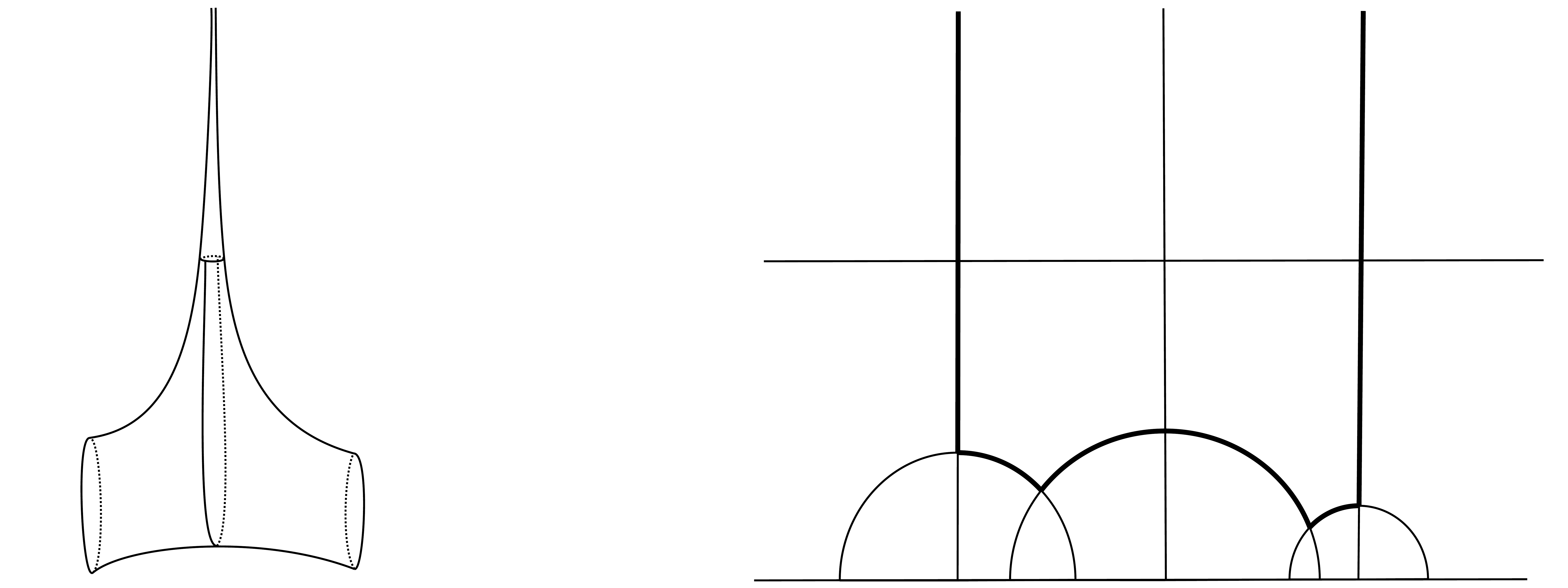}
\caption{The shortest orthogonal distance from an embedded horocycle of length $h$ to itself on a pair of pants with one cusp.}
\label{FigurePictureOrthogonalDistancev2}
\end{figure}
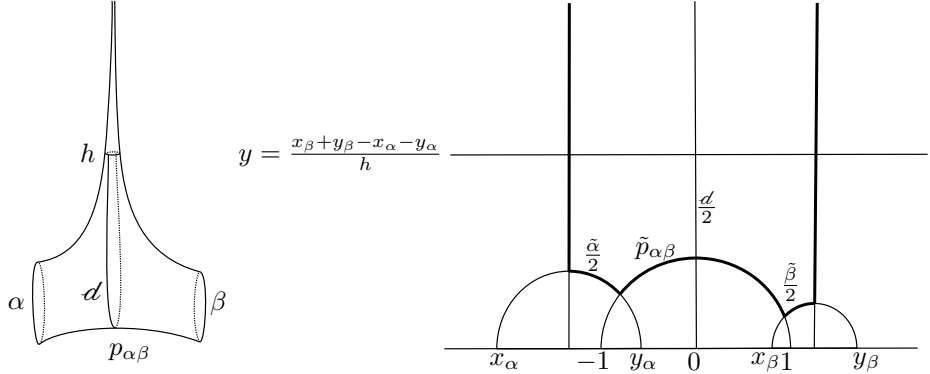

\begin{proof}
Let $p_{\alpha\beta}$ be the geodesic segment that is perpendicular to both $\alpha$ and $\beta$.
Upon lifting to the universal cover, we can normalize so that $p_{\alpha \beta}$ lifts to a geodesic segment that is contained on the Euclidean semicircle centered at the origin with radius one.
The lifts of $\alpha$ and $\beta$ are geodesic segments $\tilde{\alpha}$ and $\tilde{\beta}$ that are contained on Euclidean semicircles orthogonal to the real axis.
Denote the endpoints of $\tilde{\alpha}$ and $\tilde{\beta}$ by $x_\alpha, y_\alpha$ and $x_\beta, y_\beta$ respectively (see Figure \ref{FigurePictureOrthogonalDistancev2}).
Then the lift of the horocycle of length $h$ is the Euclidean line $y = \frac{x_\beta + y_\beta - x_\alpha - y_\alpha}{h}$.
Parametrizing $\frac{\tilde{\alpha}}{2}$ and $\frac{\tilde{\beta}}{2}$ as in the proof of Lemma \ref{LemmaBasmajiangeneralization}, by computation
$$x_\alpha = -e^{\frac{-\alpha}{2}}, \quad
y_\alpha = -e^{\frac{\alpha}{2}}, \quad
x_\beta = e^{\frac{-\beta}{2}}, \quad
y_\beta = e^{\frac{\beta}{2}}.$$
Thus,
$$\mathcal{d}(h, \alpha, \beta) = 2 \log \frac{x_\beta + y_\beta - x_\alpha - y_\alpha}{h} 
= 2 \log \frac{2 (\cosh \frac{\alpha}{2} + \cosh \frac{\beta}{2} ) }{h}, $$ 
proving the first part of the lemma.
The remaining parts follow similarly.
\end{proof}

We now recall Bers' theorem in the non-compact case (see \cite[Theorem 5.2.6]{MR2742784}). 
If $S$ is not the thrice punctured sphere, then there exists a pants decomposition of $S$ such that the simple closed geodesics are bounded above by
$$L(g,n) := 4(3g-3+n) \log \frac{4 \pi (2g-2+n)}{3g-3+n}.$$
Using Bers's constant and the previous lemma, we can find an upper bound for the shortest orthogonal distance from an embedded horocycle in $S$ to itself.
 
\begin{corollary}\label{corollaryOrthogonalDistance}
The shortest orthogonal distance from an embedded horocycle of length $h$ in $S$ to itself is at most
$$\mathcal{d}(h, g, n) = 2 \log \frac{4 \cosh \frac{L(g,n)}{2}}{h}. $$ 
\end{corollary}

\begin{proof}
This is a direct consequence of Lemma \ref{Section4LemmaOrthogonalDistance}.
\end{proof}

We can now prove Corollary \ref{Theorem2} by using the same reasoning as the one given in the proof of Theorem \ref{Theorem1}, with the constants depending on $g,n$ and $\mathcal{s}$.

\begin{proof}[Proof of Corollary \ref{Theorem2}]
Let $S \in \mathcal{M}_{g,n}^{\mathcal{s}}$.
A result by Adams \cite{MR1647904} states that the length of the longest embedded horocycle in $S$ is bounded above by 
$$\mathcal{h} = \mathcal{h}(g,n) := 12g+5n-11.$$
By our convention, the systole of $S$ is at least $\mathcal{s}$.
Set 
$$\mathcal{h_0} = \mathcal{h}_0(g,n) := \frac{1}{30 \mathcal{h}},$$
and $$\varepsilon_0 = \varepsilon_0(g,n,\mathcal{s}) := \min \left\{ \frac{\mathcal{h_0}}{5}, \frac{\mathcal{s}}{2} \right\}.$$
Denote by $S_{\varepsilon_0T}$ the $\varepsilon_0$-thick part of $S$, and $S_{\varepsilon_0t}$ the $\varepsilon_0$-thin part of $S$.
By the choice of $\varepsilon_0$, the $\varepsilon_0$-thin part $S_{\varepsilon_0t}$ is the union of embedded cusp neighborhoods. 
Each boundary component of $S_{\varepsilon_0t}$ is indeed an embedded horocycle boundary of a cusp of length 
$\frac{2}{\sqrt{\coth \varepsilon_0 - 1}}.$
We can find the shortest orthogonal distance from each boundary component of $ S_{\varepsilon_0t}$ to itself, which by Corollary \ref{corollaryOrthogonalDistance} is at most 
$$\mathcal{d}(g, n, \mathcal{s}) = 2 \log \left(2  \sqrt{\coth \varepsilon_0 - 1} \cosh \frac{L(g,n)}{2}\right).$$
Choose $D(h,g,\mathcal{s}) \ge 2$ such that
\begin{equation*}
\frac{\varepsilon_0}{12} \sqrt{x} > 2 \arsinh \frac{\varepsilon_0 x}{2} + \mathcal{d}(g, n, \mathcal{s}) + 2 \varepsilon_0, \quad \text{ for all } x \ge D(h,g,\mathcal{s}).
\end{equation*}
Set 
\begin{equation*}
\varepsilon = \varepsilon(g,n,\mathcal{s}) := \frac{\varepsilon_0}{10 D(h,g,\mathcal{s})}.
\end{equation*}
Denote by $S_{\varepsilon T}$ the $\varepsilon$-thick part of $S$, and $S_{\varepsilon t}$ the $\varepsilon$-thin part of $S$. 
By Corollary \ref{corollaryOrthogonalDistance}, the shortest orthogonal distance from the length one horocycle boundary of a cusp to the horocycle itself is at most
$$d(g, n) = 2 \log \left(4 \cosh \frac{L(g,n)}{2}\right). $$ 
Choose $K(g,n,\mathcal{s}) \ge D(g,n,\mathcal{s})$ such that 
\begin{equation*}
C(k,g,n) := 2 \arsinh(k) + d(g,n) + 1 < \frac{\varepsilon}{\mathcal{h}} \sqrt[\leftroot{0}\uproot{2}5]{k},  
\quad \text{ for all } k\ge K(g,n,\mathcal{s}).
\end{equation*}
Fix once and for all the constants $\varepsilon, \varepsilon_0, \mathcal{h}_0, D(g,n,\mathcal{s})$ and $K(g,n,\mathcal{s})$.
For all $k \ge K(g,n,\mathcal{s})$, we consider $\mathcal{G}_{\ge k}(S)$, the set of closed geodesics on $S$ that self-intersect at least $k$ times, and investigate those of minimal length.
The procedure now is the same as in the proof of Theorem \ref{Theorem1}. 
\end{proof}

In what follows, we show that the lower bound $\mathcal{s}$ on the systole lengths is necessary to make sure that the shortest geodesics are $ba_c^k$ geodesics.

\begin{example}
Given $g, n$, if there is no lower bound on systole lengths, we can construct a surface $S$ of genus $g$ with $n$ cusps such that for $k \ge 100$, shortest geodesics in $\mathcal{G}_{\ge k}(S)$ are not $ba_c^k$ geodesics.

If $g=0$, our construction works for $n \ge 9$. If $g = 1$, it works for all $n \ge 3$. If $g \ge 2$, it works for all $n \ge 1$.
What we really use in the condition of $(g,n)$ is that, if we decompose $S$ into pants, there is 
a pair of pants without cusp whose boundary lengths can be made as small as we want.

Before doing the construction, we first need some formulas.  
Let $Y(a,b,c)$ be a pair of pants with boundary geodesics of lengths $a, b$ and $c$.
Let $\alpha$ and $\beta$ be generators of the fundamental group of $Y$ which wind counter-clockwise around two boundary geodesics of lengths $a$ and $b$ respectively.
Consider the geodesic in the free homotopy class of $\beta^{-1} \alpha^{k}$.
Its length, which can be computed by using the crossed right-angled hyperbolic hexagon formula, is given by
\begin{equation}\label{EqExampleFormulaLengthBackGeodesic}
\cosh \frac{\ell (\beta^{-1} \alpha^{k})}{2}  
=  \frac{\sinh \frac{k a}{2}}{\sinh \frac{a}{2} } \left( \cosh \frac{c}{2} + \cosh \frac{a}{2}\cosh \frac{b}{2} \right) 
+ \cosh \frac{k a}{2} \cosh \frac{b}{2}.
\end{equation}
This geodesic is somewhat similar to $ba_c^k$ geodesics for the case of cusp(s).
In what follows, we will need to consider multiple geodesics of this type. 
Therefore, to simplify, we will call them $back$ geodesics. 

We will also need to use the Collar Theorem. It says that small geodesics have large tubular neighborhoods, which are topological cylinders \cite[Theorem 4.1.1]{MR2742784}.
For a simple closed geodesic $\gamma$ (of length $\gamma$) on $S$, its collar is of width
\begin{equation} \mathcal{w}(\gamma) = \arsinh \frac{1}{\sinh \frac{\gamma}{2}}.\end{equation}

Now let $Y_1 (x,x,x)$ be the pair of pants with boundary geodesics of the same length $x$.
Let $\mathcal{g}_1$ be a $back$ geodesic of $Y_1$.
By equation \eqref{EqExampleFormulaLengthBackGeodesic}, its length is given by
\begin{equation}
\cosh \frac{\ell ( \mathcal{g}_1 ) }{2} = \frac{\sinh \frac{kx}{2}}{\sinh \frac{x}{2} } \left( \cosh \frac{x}{2} + \left(\cosh \frac{x}{2}\right)^2 \right) + \cosh \frac{kx}{2} \cosh \frac{x}{2}.
\end{equation}
Let $Y_2(0,x,y)$ be the pair of pants with one cusp and two boundary geodesics of lengths $x,y$ ($x<y$).
Let $\mathcal{g}_2$ be the shortest $ba_c^k$ geodesic of $Y_2$, which is the one that winds around the cusp $k$ times and that winds around the boundary component of length $x$ once.
Its length is given by
$$\cosh \frac{\ell ( \mathcal{g}_2 ) }{2} = k \left( \cosh \frac{x}{2} + \cosh \frac{y}{2} \right) +  \cosh \frac{x}{2}.$$
Let $Y_3(0,0,y)$ be the pair of pants with two cusps and one boundary geodesic of length $y$.
Let $\mathcal{g}_3$ be the shortest $ba_c^k$ geodesic of $Y_3$, which is the one that winds $k$ times around one cusp and one time around the other cusp.
Its length is given by
$$\cosh \frac{\ell ( \mathcal{g}_3 ) }{2} = k \left( \cosh \frac{y}{2} + 1 \right) + 1
< \cosh \frac{\ell ( \mathcal{g}_2 ) }{2}.$$

\noindent
Set
$$y = 2 \arsinh \frac{1}{\sinh \frac{k}{16} }, \quad x = \frac{y}{\sqrt{2k+1} }.$$ 
With these values of $x$ and $y$, the following conditions hold for all $k \ge 100$,
\begin{itemize}
\item $\ell ( \mathcal{g}_1 ) < \ell ( \mathcal{g}_3 )$, 
\item $\ell ( \mathcal{g}_1 ) < 2 \mathcal{w}(y)$.
\end{itemize}

We can now construct $S$.
For $g=0$, we first glue $Y_1$ to three copies of $Y_2$ along the boundary geodesics of length $x$, and then glue each $Y_2$ to a copy of $Y_3$ along the boundary geodesics of length $y$
(see Figure \ref{picturegenus0v5} for $S \in \mathcal{M}_{0,9}$).
To obtain $S \in \mathcal{M}_{0,n}$ with $n \ge 10$, we can replace a copy of $Y_3$ by a surface $S^\prime$ of genus $0$ and $n-7$ cusps with systole length $y$.

\begin{figure}[ht]
\labellist
\small\hair 2pt
\pinlabel ${Y_1}$ at 560 340
\pinlabel ${Y_2}$ at 750 340
\pinlabel ${Y_3}$ at 600 700
\endlabellist
\includegraphics[width=0.45\textwidth]{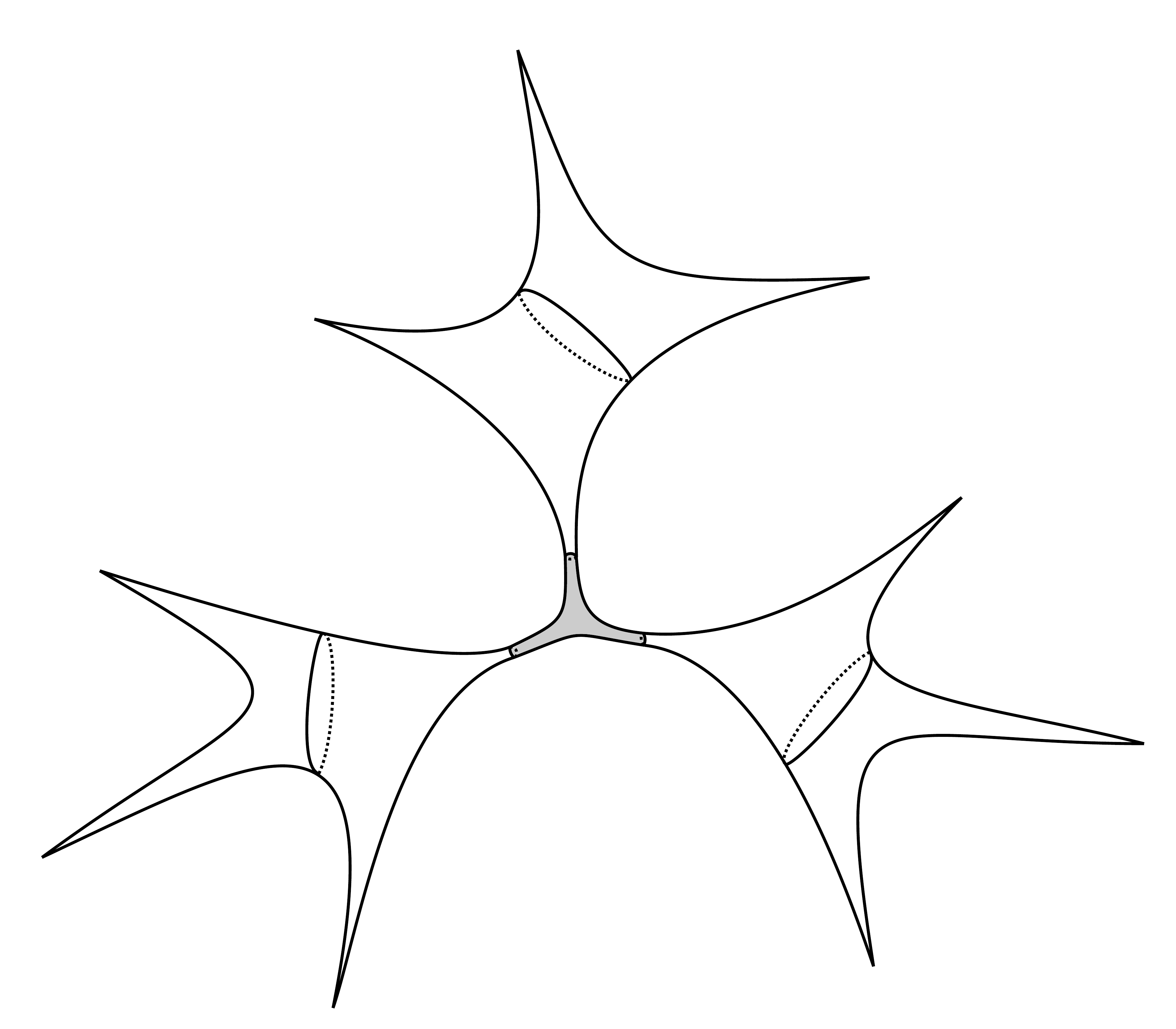}
\caption{$S \in \mathcal{M}_{0,9}$.}
\label{picturegenus0v5}
\end{figure}

For $g=1$, we first glue two boundary components of $Y_1$ together to get a one-holed torus, then glue the one-holed torus to a copy of $Y_2$, along the boundary geodesic of length $x$.
Finally, we glue the obtained surface to a surface $S^\prime$, of genus $0$ and $n-1$ cusps with systole length $y$, along the boundary geodesic of length $y$ 
(see Figure \ref{picturegenus1v6} for $S \in \mathcal{M}_{1,3}$).

\begin{figure}[ht]
\labellist
\small\hair 2pt
\pinlabel ${Y_1}$ at 10 80
\pinlabel ${Y_2}$ at 480 270
\pinlabel ${Y_3}$ at 820 270
\endlabellist
\includegraphics[width=0.35\textwidth]{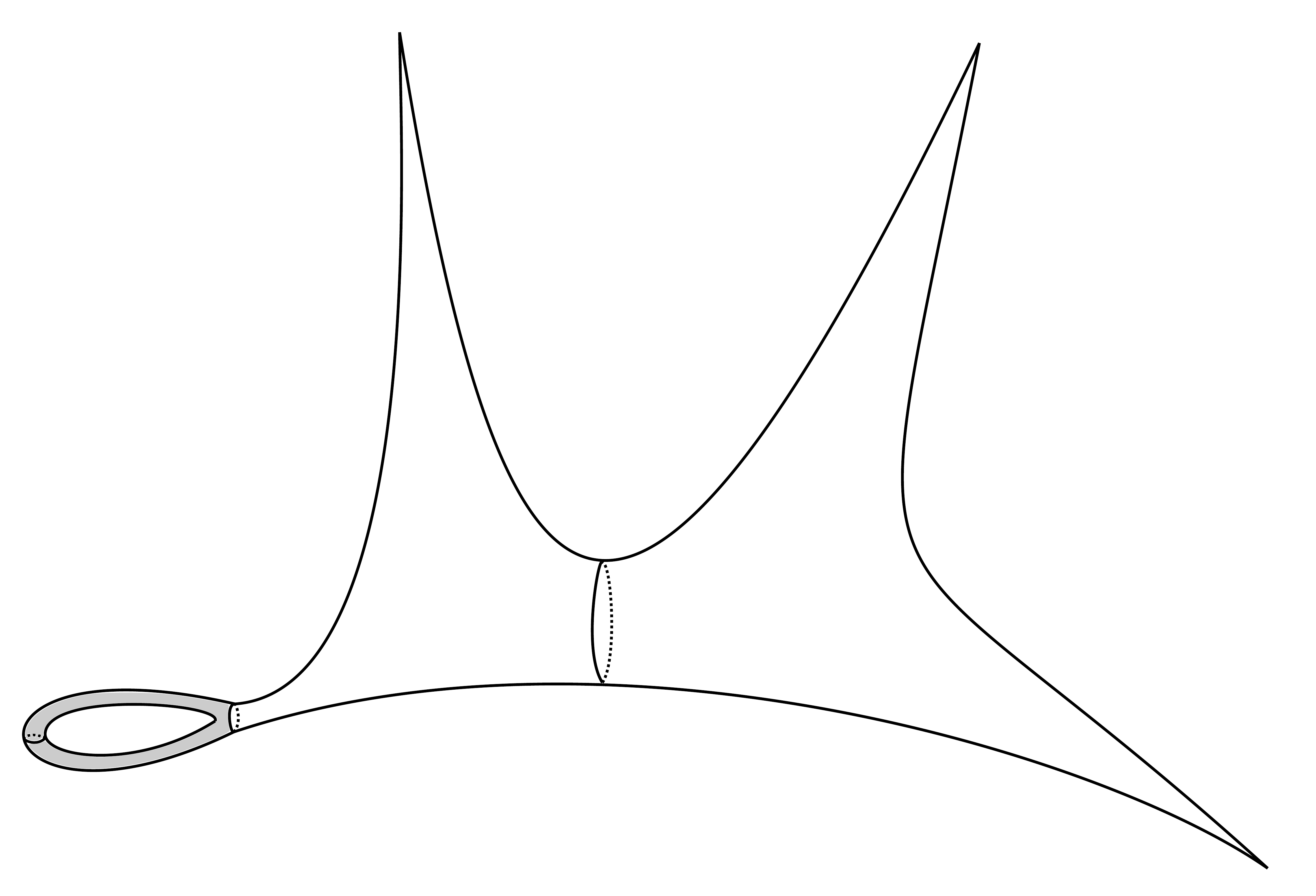}
\caption{$S \in \mathcal{M}_{1,3}$.}
\label{picturegenus1v6}
\end{figure}

For $g \ge 2$, we first glue two boundary components of $Y_1$ together to get a one-holed torus, then glue the one-holed torus to $Y_2$, along the boundary geodesic of length $x$.
Finally, we glue the obtained surface to a surface $S^\prime$, of genus $g-1$ and $n-1$ cusps with systole length $y$, along the boundary geodesic of length $y$ (see Figure \ref{picturegenus2v1} for $S \in \mathcal{M}_{2,1}$).

\begin{figure}[ht]
\labellist
\small\hair 2pt
\pinlabel ${Y_1}$ at 0 0
\pinlabel ${Y_2}$ at 540 180
\endlabellist
\includegraphics[width=0.35\textwidth]{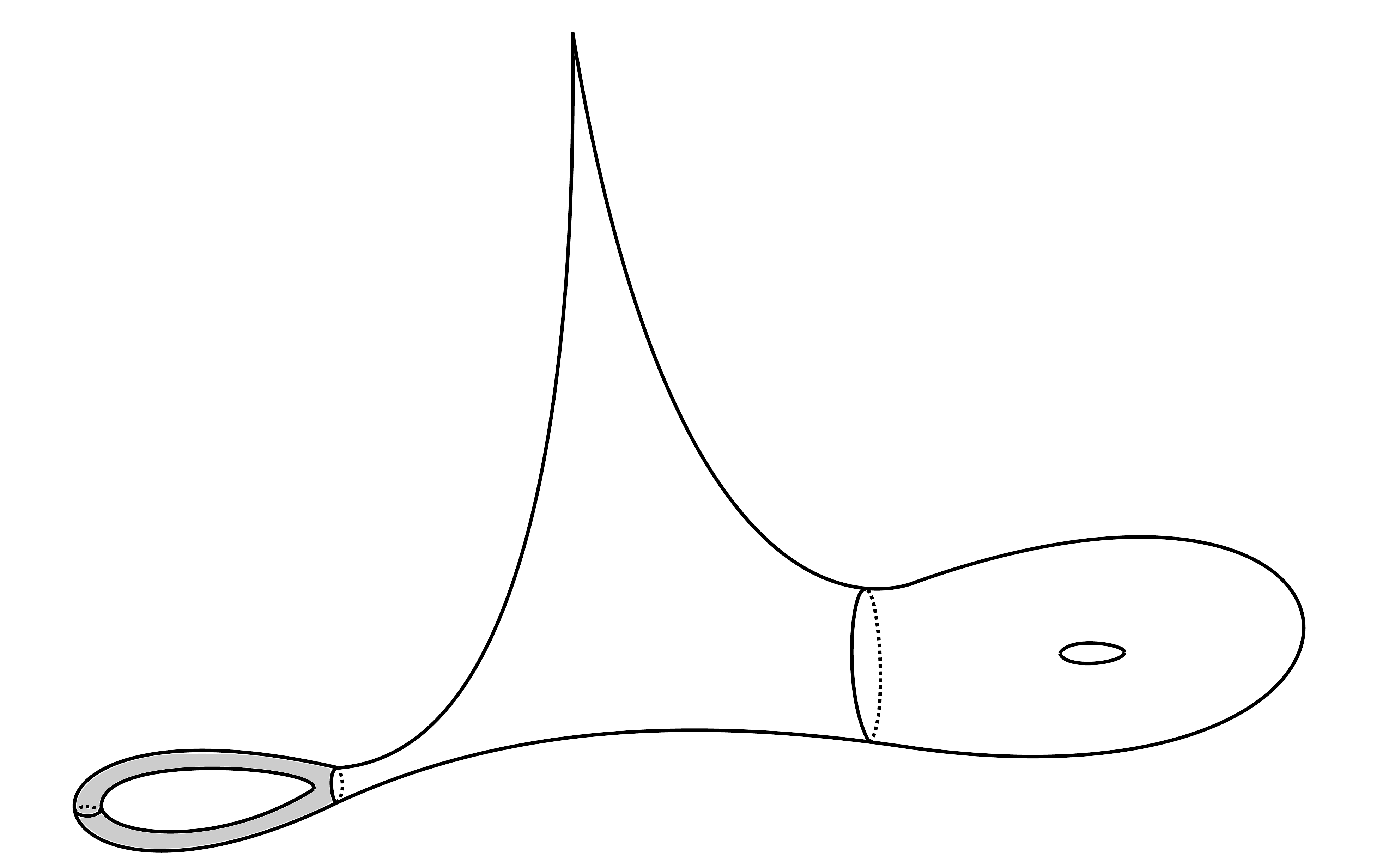}
\caption{$S \in \mathcal{M}_{2,1}$.}
\label{picturegenus2v1}
\end{figure}

We now prove that for all $k \ge 100$, the geodesic $\mathcal{g}_1$ is shorter than all $ba_c^k$ geodesics on $S$.
Let $\mathcal{g}$ be a $ba_c^k$ geodesic on $S$.
We know that $\mathcal{g}$ stays on a pair of pants with at least one cusp.
If $\mathcal{g}$ stays on a copy of $Y_2$, $Y_3$ or on $S^\prime$, then 
$$\ell(\mathcal{g}) \ge \ell(\mathcal{g}_3) > \ell(\mathcal{g}_1).$$
Otherwise, $\mathcal{g}$ must pass through a collar of width at least $\mathcal{w}(y)$, and then return. Therefore,
$$ \ell(\mathcal{g}) > 2 \mathcal{w}(y) > \ell ( \mathcal{g}_1 ).$$ 
Thus, shortest geodesics in $\mathcal{G}_{\ge k} (S)$ are not $ba_c^k$ geodesics. 
\end{example}

\begin{remark}\label{RemarkSectionCorollaryThricePuncturedSphere}
Denoting by $Y$ the thrice punctured sphere, for any non-negative integer $k$, we consider $\mathcal{G}_{\ge k}(Y)$, the set of all closed geodesics on $Y$ with at least $k$ self-intersection number.  
Let $\gamma_k$ be a shortest geodesic in $\mathcal{G}_{\ge k}(Y)$.
For $k \in \{ 0, 1 \}$, $\gamma_k$ self-intersects once. 
We want to know $i(\gamma_k, \gamma_k)$ for $k \notin \{0,1\}$.
By computing the constants defined in the proof of Theorem \ref{Theorem1}, we have  $i(\gamma_k, \gamma_k) = k$ for all $k \ge 10^{35}$ 
(here we compute the constant $K(Y)$ directly, meaning not via $D(Y)$, see Remark \ref{remarkKXdirectly} for more details).

For $2 \le k < 10^{35}$, we might expect to find $i(\gamma_k, \gamma_k)$ by a combinatorial method.
Recall that each closed geodesic on $Y$ determines a so-called \textit{word length}, which is the number of letters required to express the geodesic as a \textit{reduced cyclic word} in terms of the generators of the fundamental group and their inverses. 
The relation between the word length and the self-intersection number of a geodesic on the thrice punctured sphere has been investigated by Chas and Phillips \cite{MR2904905}.
Since we can put an upper bound on the hyperbolic length of the shortest geodesics in $\mathcal{G}_{\ge k}(Y)$, we can also put an upper bound on their word lengths. 
Thus, to find $i(\gamma_k, \gamma_k)$, the aim is to look at the set of all reduced cyclic words up to a certain word length, 
to subsequently use the Cohen - Lustig \cite{MR895629} algorithm in order to compute their self-intersection numbers,
and to finally determine the shortest ones in $\mathcal{G}_{\ge k}(Y)$.
However, since the number of involved words is tremendous, substantial computational power is needed, which makes it difficult to get examples. 
\end{remark}

\section*{Acknowledgement}
The author would like to thank Hugo Parlier for many valuable discussions and suggestions, as well as for his support and encouragement.
Furthermore, she would also like to thank Binbin Xu for fruitful conversations and remarks. 

\bibliographystyle{amsplain}
\bibliography{references}

\end{document}